\documentclass[a4paper,12pt,reqno]{amsart}

\usepackage{amsmath,amsfonts,amsthm,amssymb,mathrsfs}
\usepackage{a4wide,amsfonts,amsmath,latexsym,amssymb,
euscript,dsfont,units,mathrsfs}

\def\EE{{\mathcal E}}
\def\HH{{\mathcal H}}
\newcommand{\bx}{{\bf x}}
\newcommand{\der}{\delta}
\newcommand{\dom}{\mbox{Dom}}

\newcommand{\hro}{\hat \rho}

\newcommand{\hsi}{\hat \sigma}
\newcommand{\hz}{\hat z}

\newcommand{\id}{\mbox{Id}}

\newcommand{\iot}{\int_{0}^{t}}

\newcommand{\ist}{\int_{s}^{t}}
\newcommand{\norm}[1]{\lVert #1\rVert}

\newcommand{\ott}{[0,T]}

\newcommand{\ro}{[-r,0]}

\newcommand{\xd}{{\bf x^{2}}}
\newcommand{\bd}{{\bf B^{2}}}
\newcommand{\xdst}{{\bf x}_{st}^{\bf 2}}
\newcommand{\bdst}{{\bf B}_{st}^{\bf 2}}

\newcommand{\bdot}{{\bf B}_{0,t-s}^{\bf 2}}

\newcommand{\xdt}{{\bf \tilde{x}^{2}}}

\newcommand{\1}{{\bf 1}}
\newcommand{\E}{{\rm E}}
\def\lpa{\langle}
\def\rpa{\rangle}
\def\sk{{\mathbb{D}}}

\newcommand{\R}{\mathbb R}
\newcommand{\N}{\mathbb N}

\newcommand{\cal}{\mathcal}

\newcommand{\cb}{\mathcal B}
\newcommand{\cac}{\mathcal C}

\newcommand{\cd}{\mathcal D}

\newcommand{\cf}{\mathcal F}

\newcommand{\cj}{\mathcal J}

\newcommand{\cn}{\mathcal N}
\newcommand{\cq}{\mathcal Q}

\newcommand{\cz}{\mathcal Z}

\newcommand{\al}{\alpha}
\newcommand{\ep}{\varepsilon}
\newcommand{\e}{\varepsilon}

\newcommand{\ga}{\gamma}
\newcommand{\ka}{\kappa}

\newcommand{\laa}{\Lambda}

\newcommand{\si}{\sigma}

\newcommand{\vp}{\varphi}

\newcommand{\lp}{\left(}
\newcommand{\rp}{\right)}
\newcommand{\lc}{\left[}
\newcommand{\rc}{\right]}
\newcommand{\lcl}{\left\{}
\newcommand{\rcl}{\right\}}

\newcommand{\fin}
{ \vspace{-0.6cm}
\begin{flushright}
\mbox{$\square$}
\end{flushright}
\noindent }

\newtheorem{theorem}{Theorem}[section]

\newtheorem{corollary}[theorem]{Corollary}

\newtheorem{definition}[theorem]{Definition}

\newtheorem{hypothesis}[theorem]{Hypothesis}
\newtheorem{lemma}[theorem]{Lemma}

\newtheorem{proposition}[theorem]{Proposition}

\theoremstyle{remark}

\begin{document}

\title{Delay equations driven by rough paths}
\author{A. Neuenkirch, I. Nourdin  \and  S. Tindel}
\begin{abstract}

In this article, we illustrate the flexibility of the
algebraic integration formalism introduced in
{\it M. Gubinelli (2004), Controlling Rough Paths, J. Funct. Anal. {\bf 216}, 86-140,} by
 establishing an
existence and uniqueness result for delay equations driven by rough paths.
We then apply our results to the case where the driving path is
a fractional Brownian motion with Hurst parameter $H>\frac13$.
\end{abstract}



\address{
{\it A. Neuenkirch:}
{\rm Johann Wolfgang Goethe-Universit{\"a}t Frankfurt,
FB 12 Institut  f\"ur Mathematik,
Robert-Mayer-Strasse 10, 60325 Frankfurt am Main, Germany}.
{\it Email: }{\tt neuenkir@math.uni-frankfurt.de}. {\it Supported by the
DFG-project "Pathwise Numerics and Dynamics for Stochastic Evolution Equations".}
\newline
$\mbox{ }$\hspace{0.1cm}
{\it Ivan Nourdin:}
{\rm Laboratoire de Probabilit\'es et Mod\`eles Al\'eatoires,
Universit{\'e} Pierre et Marie Curie,
Bo{\^\i}te courrier 188, 4 Place Jussieu, 75252 Paris Cedex 5, France}.
{\it Email: }{\tt nourdin@ccr.jussieu.fr}
\newline
$\mbox{ }$\hspace{0.1cm}
{\it Samy Tindel:}
{\rm Institut {\'E}lie Cartan Nancy, Universit\'e de Nancy 1, B.P. 239,
54506 Vand{\oe}uvre-l{\`e}s-Nancy Cedex, France}.
{\it Email: }{\tt tindel@iecn.u-nancy.fr}
}

\maketitle

\section{Introduction}

In the last years, great efforts have been made to develop
a stochastic calculus for fractional Brownian motion. The first results
gave a rigorous theory for  the stochastic
integration  with respect to fractional Brownian motion and  established
a corresponding  It\^o formula, see e.g.
\cite{ALN, al,cc,  LD, nual-cours}. Thereafter, stochastic
differential equations driven by fractional Brownian motion have been
considered. Here different approaches can be used depending on   the
dimension of the equation and the Hurst parameter of the driving
fractional Brownian motion.
In the one-dimensional case \cite{nourdin-simon}, existence and
uniqueness of the solution can be derived
by a regularization procedure introduced in \cite{RV}.
The case of a multi-dimensional driving fractional Brownian motion   can
be treated  by means of fractional calculus tools, see e.g. \cite{NR,ZA} or by
means of the Young integral \cite{Lyons}, when the  Hurst coefficient  satisfies $H>\frac12$.
However,  only the rough paths theory \cite{Lyons,LyonsBook}
and its application to fractional Brownian motion \cite{CQ} allow  to solve fractional {\small SDE}s  in any dimension for a Hurst parameter $H>\frac14$.  The original
rough paths theory developed by T. Lyons relies on deeply involved  algebraical and analytical tools. Therefore
 some alternative methods \cite{FP,Gu} have been developed recently,
 trying  to catch the essential
results of \cite{LyonsBook} with less theoretical apparatus.

\vspace{0.3cm}

Since  it is
based on some rather simple algebraic considerations and
an extension of Young's integral,  the  method given in  \cite{Gu}, which we call {\it algebraic integration} in the sequel,
has been especially attractive to us. Indeed, we think
that  the basic properties of fractional differential systems can be
studied in a  natural and nice way using algebraic integration.
(See also  \cite{NNRT}, where this approach is used to study the law
of the solution of a fractional {\small SDE}.)
In the present article, we will illustrate the flexibility of the
algebraic integration formalism by studying fractional equations {\it with delay}.
More specifically, we will consider the following  equation:
\begin{equation}
\label{eq:intro-delay-b}
\left\{
\begin{array}{ll}
X_t=\xi_0+  \int_0^t \sigma(X_{s}, X_{s-r_1}, \ldots, X_{s-r_k})  dB_s
+ \int_0^t b    (X_{s}, X_{s-r_1}, \ldots,  X_{s-r_k})
ds, \quad  t\in\ott,    \\
X_t=\xi_t,  \qquad  t \in [-r_k,0].
\end{array}
\right.
\end{equation}
Here the discrete delays satisfy $0<r_1< \ldots < r_k< \infty$, the initial condition $\xi$ is a
function from $[-r_k,0]$ to $\R^n$, the functions
$\si:\R^{n,k+1} \to \R^{n,d}$, $b:\R^{n,k+1}\to \R^{n}$ are regular, and
$B$ is a $d$-dimensional fractional Brownian motion with Hurst parameter $H>\frac13$. The stochastic integral in
equation (\ref{eq:intro-delay-b}) is  a generalized Stratonovich
integral,  which will be explained in detail in Section \ref{sec:algebraic-intg}. Actually,
in equations like (\ref{eq:intro-delay-b}), the drift term $\int_0^t b    (X_{s}, X_{s-r_1}, \ldots,  X_{s-r_k})
ds$ is usually harmless, but causes some cumbersome notations. Thus, for sake of
simplicity, we will rather deal in the sequel with delay equations of the type
\begin{equation}
\label{eq:intro-delay}
\left\{
\begin{array}{ll}
X_t=\xi_0+  \int_0^t \sigma(X_{s}, X_{s-r_1}, \ldots, X_{s-r_k})  dB_s,
\quad  t\in\ott,   \\
X_t=\xi_t,  \qquad  t \in [-r_k,0].
\end{array}
\right.
\end{equation}
Our main result will be as follows:

\medskip

\begin{theorem}
Let  $\xi \in C^1([-r_k,0]; \R^{n})$,
$\si\in C_{b}^{3}(\R^{n,k+1}; \R^{n,d})$,
and let $B$ be  a $d$-di\-men\-sio\-nal fractional Brownian motion with Hurst
parameter $H>\frac13$.
Then equation (\ref{eq:intro-delay}) admits a unique solution on $\ott$ in
the class of controlled processes (see Definition \ref{def:ccp}.)
\end{theorem}

\medskip

Stochastic delay equations driven by  standard Brownian motion have been studied
extensively  (see e.g. \cite{mohammed} and \cite{mohammed2}  for an
overview) and are used in many applications. However, delay equations
driven by fractional Brownian motion have been only considered so far in  \cite{FR},
where the {\it one-dimensional} equation
\begin{equation}\label{easy}
\left\{
\begin{array}{ll}
X_t=\xi_0+  \int_0^t \sigma(X_{s-r})  dB_s
+ \int_0^t b    (X_{s})                  ds, \qquad   t\in\ott,   \\
X_t=\xi_t, \qquad  t \in [-r,0],
\end{array}
\right.
\end{equation}
is studied for $H>\frac12$. Observe that (\ref{easy}) is a particular case of equation (\ref{eq:intro-delay}).

\vspace{0.3cm}

To solve equation (\ref{eq:intro-delay}), one requires two main ingredients in the
algebraic integration setting. First of all, a natural class of paths, in which the
equation can be solved.  Here, this will be the paths whose increments are controlled
by the increments of $B$. Namely, writing $(\der z)_{st}=z_t-z_s$ for the increments
of an arbitrary  function $z$, a stochastic differential equation driven by $B$ should be solved
in the class of paths, whose increments can be decomposed into
$$
z_{t} -z_s = \zeta_s (B_{t} - B_s) +   \rho_{st},
\quad \mbox{ for } \quad 0\le s <t \le T,
$$
with $\zeta$ belonging to $\mathcal{C}_1^\ga$ and $\rho$ belonging to $\mathcal{C}_2^{2\ga}$,
for a given $\gamma\in(\frac13,H)$.
(Here, $\mathcal{C}_i^\mu$ denotes a space of $\mu$-H\"older continuous functions
of $i$ variables, see Section \ref{sec:algebraic-intg}.)
This class of functions will be called the class of
{\it controlled paths} in the sequel.

To solve  fractional differential equations {\it without} delay, the second main tool
would be to define  the integral of a  controlled
path with respect to  fractional Brownian motion and to show that the resulting
process is still a controlled path. To define the integral of a controlled path,
a double iterated integral of fractional Brownian motion,
called the {\it L\'evy area}, will be required.
Once the stability of the class of controlled paths under  integration is established, the differential
equation
is solved by an appropriate fixed point argument.

To solve  fractional delay equations,  we will have to modify this procedure. More specifically, we need a second class of paths, the class of {\it delayed controlled  paths}, whose
increments can be written as
$$
z_{t}-z_s
= \zeta_s^{(0)} ( B_{t}-B_s) +
 \sum_{i=1}^{k}\, \zeta_s^{(i)}\, (B_{t-r_i}-B_{s-r_i})
+ \rho_{st},
\quad\mbox{ for }\quad 0\le s <t \le T,
$$
where, as above, $\zeta^{(i)}$ belongs to $\mathcal{C}_1^\ga$ for $i=0, \ldots, k$,
and $\rho$ belongs to $\mathcal{C}_2^{2\ga}$ for a given $\frac13<\ga<H$. (Note that a classical controlled path is a delayed controlled path with $\zeta^{(i)}=0$ for $i=1, \ldots, k$.)  For such a delayed controlled path we will then define its integral with respect to fractional Brownian motion.
We emphasize the fact that the integral of  a delayed controlled path  is actually a classical controlled path and satisfies a stability property.

To define this integral  we have to  introduce  a {\it delayed L\'evy area}
  $\bd(v)$ of $B$ for $v \in [-r_k,0]$.  This process, with values in the space
of matrices $\R^{d,d}$ will also be defined as an iterated integral: for $1\le i,j\le d$
and $0\le s <t \le T$, we set
$$
\bdst(v)(i,j)=\ist dB_u^i  \int_{s+v}^{u+v} dB_w^j
=\int_s^t (B^j_{u+v}-B^j_{s+v}) d^\circ B^i_u,
$$
where the integral on the right hand side is a
Russo-Vallois integral \cite{RV}.
Finally, the fractional delay equation (\ref{eq:intro-delay}) will be solved by a fixed point argument.

\vspace{0.3cm}

This article is structured as follows:
Throughout the remainder of this  article, we consider the general  delay equation
\begin{equation} \label{dde}
\left\{\begin{array}{rcll}
dy_t&=& \sigma(y_{t}, y_{t-r_1}, \ldots, y_{t-r_k})  dx_t ,  & \qquad t\in\ott,    \\
{} \,\,\, y_t&=&\xi_t, &  \qquad t \in [-r_k,0],
\end{array}\right.
\end{equation}
where $x$ is $\gamma$-H\"older continuous function with $\gamma >\frac13$
 and $\xi$ is a $2\gamma$-H\"older continuous function.
In Section 2 we recall some basic facts of the algebraic integration and in particular the
definition of a classical controlled path, while in Section 3
we introduce the class of delayed controlled paths and the integral
of a delayed controlled path with
respect  to its controlling rough path. Using the stability of the integral, we show the
existence of a unique solution of equation (\ref{dde}) in the class of classical controlled paths
under the assumption of the existence of a delayed L\'evy area. Finally, in Section 4 we
specialize our results to delay equations driven by a fractional
Brownian motion with Hurst parameter $H>\frac13$.


\section{Algebraic integration and rough paths equations}\label{sec:algebraic-intg}
Before we consider equation (\ref{dde}), we recall the  strategy
introduced in  \cite{Gu}  in order to solve an equation without delay, i.e.,
\begin{equation} \label{de_rp}
dy_t = \sigma(y_t)  d x_t, \quad t \in [0,T], \quad  \qquad y_0= \alpha \in \R^n,
\end{equation}
where $x$ is a $\R^d$-valued $\gamma$-H\"older continuous function with
$\gamma > \frac13$.

\subsection{Increments}\label{incr}

Here we present the basic  algebraic structures, which
will allow us to define a pathwise integral with respect to
irregular functions. For  real numbers
$0 \leq a \leq b \leq T < \infty $, a vector space $V$ and an integer $k\ge 1$ we denote by
$\cac_k([a,b]; V)$ the set of functions $g : [a,b]^{k} \to V$ such
that $g_{t_1 \cdots t_{k}} = 0$
whenever $t_i = t_{i+1}$ for some $1 \leq i\le k-1$.
Such a function will be called a
\emph{$(k-1)$-increment}, and we will
set $\cac_*([a,b];V)=\cup_{k\ge 1}\cac_k([a,b];V)$. An important  operator for our purposes
is given by
\begin{equation}
  \label{eq:coboundary}
\delta : \cac_k([a,b];V) \to \cac_{k+1}([a,b];V), \qquad
(\delta g)_{t_1 \cdots t_{k+1}} = \sum_{i=1}^{k+1} (-1)^{k-i}
g_{t_1  \cdots \hat t_i \cdots t_{k+1}} ,
\end{equation}
where $\hat t_i$ means that this argument is omitted.
A fundamental property of $\der$
is that
$\delta \delta = 0$, where $\delta \delta$ is considered as an operator
from $\cac_k([a,b];V)$ to $\cac_{k+2}([a,b];V)$.
 We will denote $\cz\cac_k([a,b];V) = \cac_k( [a,b];V) \cap \text{Ker}\delta$
and $\cb \cac_k([a,b];V) =
\cac_k([a,b];V) \cap \text{Im}\delta$.

\vspace{0.3cm}

Some simple examples of actions of $\der$
 are as follows: For
$g\in\cac_1([a,b];V)$, $h\in\cac_2([a,b];V)$ and $f\in\cac_3([a,b];V)$ we have
\begin{equation*}
  (\der g)_{st} = g_t - g_s,
\quad
(\der h)_{sut} = h_{st}-h_{su}-h_{ut}\quad\mbox{and}\quad
(\der f)_{suvt}=f_{uvt}-f_{svt}+f_{sut}-f_{suv}
\end{equation*}
for any $s,u,v,t\in [a,b]$.
Furthermore, it is easily checked that
$\cz \cac_{k+1}([a,b];V) = \cb \cac_{k}([a,b];$ $V)$ for any $k\ge 1$.
In particular, the following  property holds:
\begin{lemma}\label{exd}
Let $k\ge 1$ and $h\in \cz\cac_{k+1}([a,b];V)$. Then there exists a (non unique)
$f\in\cac_{k}([a,b];V)$ such that $h=\der f$.
\end{lemma}

Observe that Lemma \ref{exd} implies in particular that all  elements
$h \in\cac_2([a,b];V)$  with $\der h= 0$ can be written as $h = \der f$
for some  $f \in \cac_1([a,b];V)$. Thus we have a heuristic
interpretation of $\der |_{\cac_2([a,b];V)}$:  it measures how much a
given 1-increment  differs from being an  exact increment of a
function, i.e., a finite difference.

\vspace{0.3cm}

 Our further discussion will mainly rely on
$k$-increments with $k \le 2$.
For  simplicity of the exposition, we will assume that $V=\R^d$ in what follows,
although $V$
could be in fact  any Banach space.
We measure the size of the increments by H\"older norms,  which are
defined in the following way: for $f \in \cac_2([a,b];V)$ let
\begin{align*} 
\|f\|_{\mu} =
\sup_{s,t\in [a,b]}\frac{|f_{st}|}{|t-s|^\mu}\end{align*}
and $$
\cac_2^\mu([a,b];V)=\lcl f \in \cac_2([a,b];V);\, \|f\|_{\mu}<\infty  \rcl.
$$
Obviously, the usual H\"older spaces $\cac_1^\mu([a,b];V)$  are determined
        in the following way: for a continuous function $g\in\cac_1([a,b];V)$  set
\begin{equation*} 
\|g\|_{\mu}=\|\der g\|_{\mu},
\end{equation*}
and we will say that $g\in\cac_1^\mu([a,b];V)$ iff $\|g\|_{\mu}$ is finite.
Note that $\|\cdot\|_{\mu}$ is only a semi-norm on $\cac_1([a,b];V)$,
but we will  work  in general on spaces of the type
\begin{equation*}
\cac_{1,\al}^\mu([a,b];V)=
\lcl g:[a,b] \to V;\, g_a=\alpha,\, \|g\|_{\mu}<\infty \rcl,
\end{equation*}
for a given $\alpha \in V,$ on which $\|g\|_{\mu}$  is a norm.

 For $h \in \cac_3([a,b];V)$ we define in the same way
\begin{eqnarray}\label{eq:notation-norm-C3}
  \norm{h}_{\gamma,\rho} &=& \sup_{s,u,t\in [a,b] }
\frac{|h_{sut}|}{|u-s|^\gamma |t-u|^\rho}\\
\|h\|_\mu &= &
\inf \left \{\sum_i \|h_i\|_{\rho_i,\mu-\rho_i} ; (\rho_{i},h_{i})_{i \in \N} \textrm{ with } h_{i} \in \cac_3([a,b];V),
 \sum_i h_i =h , 0 < \rho_i < \mu  \right\}. \nonumber
\end{eqnarray}
Then  $\|\cdot\|_\mu$  is a norm on $\cac_3([a,b];V)$, see \cite{Gu}, and we define
$$
\cac_3^\mu([a,b];V):=\lcl h\in\cac_3([a,b];V);\, \|h\|_\mu<\infty \rcl.
$$
Eventually,
let $\cac_3^{1+}([a,b];V) = \cup_{\mu > 1} \cac_3^\mu([a,b];V)$
and  note that the same kind of norms can be considered on the
spaces $\cz \cac_3([a;b];V)$, leading to the definition of  the  spaces
$\cz \cac_3^\mu([a;b];V)$ and $\cz \cac_3^{1+}([a,b];V)$.

\vspace{0.3cm}

The crucial point in this algebraic approach to the  integration of irregular
paths is that the operator
$\delta$ can be inverted under mild  smoothness assumptions. This
inverse is called $\laa$. The proof of the following proposition  may be found
in \cite{Gu}, and in a simpler form in \cite{GT}.

\medskip

\begin{proposition}
\label{prop:Lambda}
There exists a unique linear map $\Lambda: \cz \cac^{1+}_3([a,b];V)
\to \cac_2^{1+}([a,b];V)$ such that
$$
\delta \Lambda  = \id_{\cz \cac_3^{1+}([a,b];V)}
\quad \mbox{ and } \quad \quad
\Lambda  \delta= \id_{\cac_2^{1+}([a,b];V)}.
$$
In other words, for any $h\in\cac^{1+}_3([a,b];V)$ such that $\der h=0$,
there exists a unique $g=\laa(h)\in\cac_2^{1+}([a,b];V)$ such that $\der g=h$.
Furthermore, for any $\mu > 1$,
the map $\laa$ is continuous from $\cz \cac^{\mu}_3([a,b];V)$
to $\cac_2^{\mu}([a,b];V)$ and we have
\begin{equation}\label{ineqla}
\|\Lambda h\|_{\mu} \le \frac{1}{2^\mu-2} \|h\|_{\mu} ,\qquad h \in
\cz \cac^{\mu}_3([a,b];V).
\end{equation}
\end{proposition}

\medskip

This mapping  $\laa$  allows to construct a generalised Young integral:
\begin{corollary}
\label{cor:integration}
For any 1-increment $g\in\cac_2 ([a,b];V)$ such that $\der g\in\cac_3^{1+}([a,b];V)$
set
$
\delta f = (\id-\Lambda \delta) g
$.
Then
$$
(\delta f)_{st} = \lim_{|\Pi_{st}| \to 0} \sum_{i=0}^n g_{t_i\,t_{i+1}}
$$
for $a\leq s < t \leq b$, where the limit is taken over any partition $\Pi_{st} = \{t_0=s,\dots,
t_n=t\}$ of $[s,t]$, whose mesh tends to zero. Thus, the
1-increment $\delta f$ is the indefinite integral of the 1-increment $g$.
\end{corollary}

\medskip

We also need some product rules for the operator $\delta$. For this
recall the following convention:
for  $g\in\cac_n([a,b];\R^{l,d})$ and $h\in\cac_m( [a,b];\R^{d,p}) $ let  $gh$
be the element of $\cac_{n+m-1}( [a,b];\R^{l,p})$ defined by
\begin{equation}\label{cvpdt}
(gh)_{t_1,\dots,t_{m+n-1}}=
g_{t_1,\dots,t_{n}} h_{t_{n},\dots,t_{m+n-1}},
\end{equation}
for $t_1,\dots,t_{m+n-1}\in [a,b].$
\begin{proposition}\label{difrul} It holds:
\begin{enumerate}
\item[{\it(i)}]
Let $g\in\cac_1([a,b];\R^{l,d})$ and $h\in\cac_1([a,b],\R^d)$. Then
$gh\in\cac_1(\R^l)$ and
\begin{equation*}
\der (gh) = \der g\,  h + g\, \der h.
\end{equation*}
\item[{\it(ii)}]
Let $g\in\cac_1([a,b]; \R^{l,d})$ and $h\in\cac_2([a,b];\R^d)$. Then
$gh\in\cac_2([a,b];\R^l)$ and
\begin{equation*}
\der (gh) = - \der g\, h + g \,\der h.
\end{equation*}
\item[{\it(iii)}]
Let $g\in\cac_2 ([a,b];\R^{l,d})$ and $h\in\cac_1([a,b];\R^d)$. Then
$gh\in\cac_2([a,b];\R^l)$ and
\begin{equation*}
\der (gh) = \der g\, h  + g \,\der h.
\end{equation*}
\item[{\it(iv)}]
Let $g\in\cac_2 ([a,b];\R^{l,d})$ and $h\in\cac_2([a,b];\R^{d,p})$. Then
$gh\in\cac_3([a,b];\R^{l,p})$ and
\begin{equation*}
\der (gh) = - \der g\, h  + g \,\der h.
\end{equation*}
\end{enumerate}
\end{proposition}

\vspace{0.3cm}


\subsection{Classical controlled paths (CCP)}\label{sec:cl4}

In the remainder of this article, we will use  both the notations $\ist f dg$
or $\cj_{st}(f\, dg)$ for the integral of a function $f$ with
respect to a given function $g$ on the interval $[s,t]$.
Moreover,  we also set
$ \| f \|_{\infty} = \sup_{ x \in \R^{d,l} } |f(x)| $
for a function $f: \R^{d,l} \rightarrow  \R^{m,n}$. To simplify the notation
we will write $\cac_k^{\gamma}$ instead of
$\cac_k^{\gamma}([a,b];V)$, if $[a,b]$ and $V$ are obvious from the context.

\vspace{0.3cm}

Before we consider the technical  details,
we will make some heuristic considerations about the properties that the solution of  equation
  (\ref{de_rp}) should enjoy. Set
$\hsi_t=\si\lp y_{t} \rp$, and suppose that $y$ is a solution of  (\ref{de_rp}),
which satisfies $y\in\cac_1^\ka$ for a given $\frac13<\ka<\ga$.
Then the integral form of our equation can be
written as
\begin{equation}\label{intg}
y_t=\alpha+\iot \hsi_u dx_u,  \qquad t\in\ott.
\end{equation}
Our approach to generalised integrals induces us to work with increments
of the form $(\delta y)_{st}=y_t-y_s$  instead  of (\ref{intg}). It is immediate that one can decompose
the increments of
(\ref{intg}) into
$$
(\delta y)_{st}=\ist \hsi_u dx_u=\hsi_s (\delta x)_{st}+\rho_{st}
\quad\mbox{ with }\quad
\rho_{st}=\ist (\hsi_u-\hsi_s) dx_u.
$$
We thus have  obtained a decomposition of $y$ of the form
$\delta y=\hsi\delta x+\rho$.  Let us see, still at a heuristic level,
 which regularity we can expect for $\hsi$ and $\rho$: If
$\si$ is bounded and continuously differentiable, we have that $\hsi$ is bounded and
$$
|\hsi_t-\hsi_s|
\le
\|\si'  \|_{\infty}  \|y\|_{\ka} |t-s|^{\ka},
$$
where $\|y\|_{\ka}$ denotes the $\ka$-H\"older norm of $y$.  Hence
 $\hsi$ belongs to $\cac_1^\ka$ and is bounded.
As far as $\rho$ is
concerned, it should
inherit both the regularities of $\delta\hsi$ and $x$,  provided that  the integral
$\ist (\hsi_u-\hsi_s) dx_u=\ist(\delta\hsi)_{su}dx_u$ is well defined.  Thus, one should
expect that $\rho\in\cac_2^{2\ka}$.  In summary, we have found that
a solution $\delta y$ of  equation   (\ref{intg}) should be decomposable into
\begin{equation}\label{first:structure1}
\delta y =\hsi \delta x + \rho
\quad\mbox{ with }\quad
\hsi\in\cac_1^\ka
\,\, \mbox{ bounded and } \,\,
\rho\in\cac_2^{2\ka}.
\end{equation}
This is precisely the structure we will  demand for a possible
solution  of
equation (\ref{de_rp}) respectively its integral form (\ref{intg}):
\begin{definition}\label{def:ccp}
Let $a \leq b \leq T$
and let $z$ be a path in $\cac_1^\ka([a,b];\R^n)$ with $\ka\le\ga$
and $2\ka+\ga>1$.
We say that $z$ is a classical controlled path based on $x$, if
$z_a= \alpha \in \R^{n}$ and $\der z\in\cac_2^\ka([a,b];\R^n)$ can be decomposed into
\begin{equation}\label{weak:dcp}
\der z=\zeta \der x+ r,
\quad\mbox{i.\!\! e.}\quad
(\der z)_{st}=\zeta_s (\der x)_{st} + \rho_{st},
\quad s,t\in [a,b],
\end{equation}
with $\zeta\in\cac_1^\ka([a,b];\R^{n,d})$ and $\rho \in \cac_2^{2\ka}([a,b];\R^n)$.\\
The space of classical controlled
paths on $[a,b]$ will be denoted by $\cq_{\ka,\alpha}([a,b];\R^n)$, and a path
$z\in\cq_{\ka,\alpha}([a,b];\R^n)$ should be considered in fact as a couple
$(z,\zeta)$.\\
The
norm on $\cq_{\ka,\alpha}([a,b];\R^n)$ is given
by
$$
\cn[z;\cq_{\ka,\alpha}([a,b];\R^n)]=
 \sup_{s,t \in [a,b]}
\frac{|(\delta z)_{st}|}{{}\, \, |s-t|^{\kappa}} +  \sup_{s,t \in [a,b]}
\frac{| \rho_{st}|}{{}\, \, |s-t|^{2\kappa}} +
 \sup_{t \in [a,b]} |\zeta_t | +  \sup_{s,t \in [a,b]}
\frac{|(\delta \zeta)_{st}|}{{}\, \, |s-t|^{\kappa}}. $$
\end{definition}

\smallskip

Note that  in the above definition $\alpha$ corresponds to a given initial condition and
$\rho$ can be understood as a regular part. Moreover, observe that $a$ can be negative.

\smallskip

Now we can
sketch the strategy used in \cite{Gu}, in order
to solve equation (\ref{de_rp}):
\begin{enumerate}
\item[(a)]
Verify the stability of $\cq_{\ka,\alpha}([a,b];\R^n)$ under a smooth map
$\vp: \R^{n} \rightarrow \R^{n,d}$.
\item[(b)]
Define rigorously the integral $\int z_u dx_u=\cj(z dx)$
for a classical controlled path $z$ and compute its decomposition
(\ref{weak:dcp}).
\item[(c)]
Solve equation (\ref{de_rp}) in the space $\cq_{\ka,\alpha}([a,b]; \R^{n})$
by a fixed point argument.
\end{enumerate}
Actually, for the second point  we had
to impose  a priori the following hypothesis  on the driving rough
  path, which is a standard  assumption in the rough paths theory:

\smallskip

\begin{hypothesis}\label{hyp:x-cl}
The $\R^d$-valued
$\ga$-H\"older path $x$ admits a L\'evy area,
 i.e. a process $\xd=\cj(dx dx)\in\cac_2^{2\ga}([0,T];\R^{d, d})$, which satisfies
$
\der\xd=\der x\otimes \der x,$
that is
$$ \lc (\der\xd)_{sut} \rc(i,j)
=
[\der x^{i}]_{su} [\der x^{j}]_{ut},
\quad \textrm{for all } \quad s,u,t\in\ott, \, i,j\in\{1,\ldots,d  \}.
$$
\end{hypothesis}

\medskip

Then, using the
  strategy sketched above, the following
  result is obtained in \cite{Gu}:

\smallskip

\begin{theorem}\label{thm:ex-uniq1}
Let $x$ be a process satisfying Hypothesis \ref{hyp:x-cl} and
let $\si\in C^2(\R^{n}; \R^{n,d})$ be  bounded
together with its derivatives. Then we have:
\begin{enumerate}
\item
Equation (\ref{de_rp}) admits a unique solution $y$ in
$\cq_{\ka,\alpha}([0,T];\R^n)$ for any $\ka<\ga$ such that $2\ka+\ga>1$.
\item
The mapping $(\alpha,x,\xd)\mapsto y$ is continuous from
$\R^n\times\cac_1^{\ga}([0,T];\R^d)\times\cac_2^{2\ga}([0,T];\R^{d, d})$
to $\cq_{\ka,\alpha}([0,T];\R^n)$, in a sense which is detailed in \cite[Proposition 8]{Gu}.
\end{enumerate}
\end{theorem}

\medskip

\section{The delay equation}

In this section, we  make a first step towards the solution of the delay equation
\begin{equation}
\begin{cases}\label{dde_2}
dy_t= \sigma(y_{t}, y_{t-r_1}, \ldots, y_{t-r_k})  dx_t, &   \qquad   t\in\ott,  \\
{} \, \, \,  y_t=\xi_t,  & \qquad  t \in [-r_k,0],
\end{cases}
\end{equation}
where  $x$ is a $\R^d$-valued $\gamma$-H\"older continuous function with
$\gamma>\frac13$, the function $\sigma \in C^3(\R^{n,k+1} ;\R^{n,d})$ is bounded
together with its derivatives, $\xi$ is a $\R^n$-valued
$2\gamma$-H\"older continuous function, and $0<r_1< \ldots < r_k< \infty$. For
convenience, we  set $r_0=0$
and, moreover, we will use the notation
\begin{equation}\label{eq:def-s(y)}
\mathfrak{s}(y)_t=(   y_{t-r_1}, \ldots, y_{t-r_k}     ), \qquad t \in \ott.
\end{equation}

\subsection{Delayed controlled paths}

As in the previous section,  we will first make some heuristic considerations about the properties of a solution:
set
$\hsi_t=\si(y_{t}, \mathfrak{s}(y)_t)      $ and  suppose that $y$ is a
solution of  (\ref{dde_2}) with $y\in\cac_1^\ka$ for a given $\frac13<\ka<\ga$.
Then we can write the  integral form of our equation as $$
(\delta y)_{st}=\ist \hsi_u dx_u=\hsi_s (\delta x)_{st}+\rho_{st}
\quad\mbox{ with }\quad
\rho_{st}=\ist (\hsi_u-\hsi_s) dx_u.
$$
Thus, we have  again obtained a decomposition of $y$ of the form
$\delta y=\hsi\delta x+\rho$.   Moreover, it follows (still at a heuristic level)
 that $\hsi$ is bounded and satisfies
$$
|\hsi_t-\hsi_s|
\le
\| \si '  \|_{\infty}
 \sum_{i=0}^{k} | y_{t-r_i} - y_{s-r_i}|
\le (k+1)
\|  \si '  \|_{\infty} \|y\|_{\ga} |t-s|^{\ga}.
$$
Thus, with the notation of Section
\ref{incr},
we have that $\hsi$ belongs to $\cac_1^\ga$ and is bounded.
 The term $\rho$ should again
inherit both the  regularities of $\delta\hsi$ and $x$. Thus, one should
have that $\rho\in\cac_2^{2\ka}$. In conclusion, the increment
$\delta y$ should be decomposable into
\begin{equation}\label{first:structure}
\delta y =\hsi \delta x + \rho
\quad\mbox{ with }\quad
\hsi\in\cac_1^\ga
\,\, \mbox{ bounded and } \,\,
\rho\in\cac_2^{2\ka}.
\end{equation}
This is again the   structure we will ask for a possible solution to
(\ref{dde_2}). However, this decomposition does not take  into
account that  equation (\ref{dde_2})  is actually a {\it delay} equation.
To define the integral $\int_s^{t} \hsi_u d x_u$, we have
to enlarge the class of functions we will
work with, and hence we will define
a {\it delayed controlled path} (hereafter {\small DCP} in short).

\medskip

\begin{definition}\label{def:dcp}
Let $0\leq a \leq b \leq T$ and
$z \in \cac_1^\ka([a,b];\R^{n})$ with $\frac13<\ka\le\ga$. We say that $z$ is a delayed controlled
path based on $x$, if
$z_{a}= \alpha$ belongs to $\R^{n}$
and if $\der z\in\cac_2^\ka([a,b];\R^{n})$ can be decomposed into
\begin{equation}\label{decompo:dcp}
(\der z)_{st}
=
\sum_{i=0}^{k}\, \zeta_s^{(i)}\, (\delta x)_{s-r_i,t-r_i}
+ \rho_{st}
\quad\mbox{ for }\quad
s,t\in [a,b],
\end{equation}
 where $\rho \in \cac_2^{2\ka}([a,b];\R^{n})$ and  $\zeta^{(i)} \in  \cac_1^{\ka}([a,b];\R^{n,d})$ for $i=0, \ldots, k$.\\
The space of delayed controlled
paths on $[a,b]$ will be denoted by $\cd_{\ka,\alpha} ([a,b];\R^{n})    $, and a path
$z\in\cd_{\ka,\alpha}    ([a,b];\R^{n})      $ should be considered in fact as a $(k+2)$-tuple
$(z,\zeta^{(0)}, \ldots , \zeta^{(k)})$.\\
The  norm on $\cd_{\ka,\alpha}([a,b];\R^{n})$ is given
by
\begin{eqnarray*}
&&\cn[z;\cd_{\ka,\alpha}([a,b];\R^{n})]=\sup_{s,t \in [a,b]}
\frac{|(\delta z)_{st}|}{{}\, \, |s-t|^{\kappa}} +  \sup_{s,t \in [a,b]}
\frac{| \rho_{st}|}{{}\, \, |s-t|^{2\kappa}}  \\ && \qquad \qquad  \qquad \qquad \qquad   \qquad \qquad
+ \sum_{i=0}^{k}\sup_{t \in [a,b]} |\zeta_t^{(i)} | +   \sum_{i=0}^{k} \sup_{s,t \in [a,b]}
\frac{|(\delta \zeta^{(i)})_{st}|}{{}\, \, |s-t|^{\kappa}}.
\end{eqnarray*}
\end{definition}


\vspace{0.3cm}

Now  we can sketch our strategy to solve the delay
equation:
\begin{enumerate}
\item
Consider the map $T_{\sigma}$ defined on $\cq_{\ka,\alpha}([a,b];\R^n) \times \cq_{\ka,\tilde{\alpha}}([a-r_k,b-r_1];\R^n) $ by
\begin{equation}
(T_{\sigma}(z, \tilde{z}))_t
=\si( z_t, \mathfrak{s}(\tilde{z})_{t}), \qquad t \in [a,b],
\end{equation}
where we recall that the notation $\mathfrak{s}(\tilde{z})$ has been introduced at (\ref{eq:def-s(y)}).
We will show that $T_{\sigma}$ maps  $\cq_{\ka,\alpha}([a,b];\R^n) \times \cq_{\ka,\tilde{\alpha}}([a-r_k,b-r_1];\R^n) $  smoothly onto
a space of the form $\cd_{\ka,\hat{\alpha}}([a,b];\R^{n,d})$.
\item
Define rigorously the integral $\int z_u dx_u=\cj(z dx)$
for a delayed controlled path $z \in \cd_{\ka,\hat{\alpha}}([a,b];\R^{n,d})$, show that $\cj(z dx)$ belongs
to $\cq_{\ka,\alpha}([a,b]; \R^{d})$, and compute its decomposition
(\ref{weak:dcp}).  Let us point out the following important fact: $T_{\sigma}$ creates ``delay'',
that is $T_{\sigma}(z, \tilde{z}) \in \cd_{\ka,\hat{\alpha}}( [a,b];\R^{n ,d})$,
while $\mathcal{J}$ creates ``advance'',
that is ${\mathcal J}(zdx)\in\cq_{\ka, \alpha}  ([a,b];\R^{n})      $.
\item
By combining the first two points, we will
solve equation (\ref{dde_2})
by a fixed point argument on the intervals $[0,r_1],[r_1, 2 r_1], \ldots $ .
\end{enumerate}

\subsection{Action of the map $T$ on controlled paths}

The  major part of this section will be devoted to the following two
stability results:

\begin{proposition}\label{compo:ccp-phi}
 Let $0 \leq a \leq b \leq T$, let $\al,\tilde\al$ be two initial conditions in $\R^n$
 and let $\vp \in C^{3}(\R^{n,k+1};\R^{l})$ be bounded with bounded derivatives.
Define $T_{\vp}$ on $\cq_{\ka,\alpha}([a;b];\R^{n})
\times \cq_{\ka,\tilde\alpha}([a-r_k;b-r_1];\R^{n})$
 by $T_{\vp} (z, \tilde{z})=\hz$, with
$$
\hz_t
=\vp (z_t, \mathfrak{s}(\tilde{z})_{t}), \qquad t \in [a,b].
$$
Then, setting
$\hat \al=\vp(\al,\mathfrak{s}(\tilde{z}_a))=\vp(\al,\tilde z_{a-r_1},\ldots,\tilde z_{a-r_{k-1}},\tilde\al)$, we have
$T_{\vp} (z, \tilde{z}) \in \cd_{\ka,\hat \al}([a;b];$ $\R^{l})$
 and it admits a decomposition
of the form
\begin{equation}\label{decompo:ttz}
\lp \delta \hz \rp_{st}
=  \hat \zeta_s \,  (\delta x)_{st} +
\sum_{i=1}^{k} \hat\zeta_s^{(i)}\,  (\delta x)_{s-r_i,t-r_i} +\hro_{st}, \qquad s,t \in [a,b], \end{equation}
 where $\hat\zeta,\hat\zeta^{(i)}$ are the $\R^{l,d}$-valued paths defined by
\begin{equation*} \hat \zeta_s=    \left(  \frac{ \partial \vp }{
      \partial x_{1,0}}  (z_s, \mathfrak{s}(\tilde{z})_{s})  , \ldots
  ,
    \frac{ \partial \vp }{ \partial x_{n,0}} (z_s, \mathfrak{s}(\tilde{z})_{s})
   \right) \zeta_{s} , \qquad  s\in [a,b],
\end{equation*}
and
\begin{equation*} \hat \zeta_s^{(i)}=    \left(  \frac{ \partial \vp }{
      \partial x_{1,i}}   (z_s, \mathfrak{s}(\tilde{z})_{s})
 , \ldots  ,  \frac{ \partial \vp }{ \partial x_{n,i}}
  (z_s, \mathfrak{s}(\tilde{z})_{s}) \right)  \tilde{\zeta}_{s-r_i} , \qquad s\in [a,b],
\end{equation*}
 for $i=1, \ldots, k.$
Moreover, the following estimate holds:
\begin{align}\label{bnd:t-phi-z}
&\cn[\hz;\cd_{\ka,\hat a}([a;b];  \R^{l})]
 \\ & \qquad \quad \le c_{\vp, T}  \lp 1+ \cn^2[z;\cq_{\ka,\alpha}([a,b];\R^{n})] +
\cn^2[\tilde z;\cq_{\ka,\tilde{\alpha}}([a-r_k,b-r_1];\R^{n})] \rp,
\nonumber
\end{align}
where the constant $c_{\vp,   T}$ depends only $\vp$ and $T$.
\end{proposition}

\begin{proof} Fix $s,t \in [a,b]$ and set
$$\psi_s^{(i)}=  \left(  \frac{ \partial \vp }{ \partial x_{1,i}}    (z_s, \mathfrak{s}(\tilde{z})_{s})
, \ldots  ,   \frac{ \partial \vp }{ \partial x_{n,i}}    (z_s, \mathfrak{s}(\tilde{z})_{s}) \right). $$
for $i=0, \ldots, k$. It is readily checked that
\begin{eqnarray*}
(\delta\hz)_{st}
&=&
\vp (  z_{t-r_0},  \tilde{z}_{t-r_1}, \ldots, \tilde{z}_{t-r_k}) -\vp (
z_{s-r_0},   \tilde{z}_{s-r_1},  \ldots, \tilde{z}_{s-r_k})
\\ &=&   \psi_s^{(0)} \zeta_{s} (\delta x)_{st} +\sum_{i=1}^{k} \psi_s^{(i)} \tilde{\zeta}_{s-r_i} (\delta x)_{s-r_i, t-r_i}
+\hro_{st}^1 +\hro_{st}^{2},
\end{eqnarray*}
where
\begin{eqnarray*}
\hro_{st}^{1}&=& \psi_s^{(0)} \rho_{st}
+ \sum_{i=1}^{k}  \psi_s^{(i)} \tilde\rho_{s-r_i,t-r_i}, \nonumber\\
\hro_{st}^2&=& \vp (  z_{t-r_0},  \tilde{z}_{t-r_1}, \ldots, \tilde{z}_{t-r_k}) -\vp (  z_{s-r_0}, \tilde{z}_{s-r_1},
\ldots, \tilde{z}_{s-r_k})  \\ &&  \qquad \qquad  \qquad \qquad\qquad \qquad \qquad \qquad
- \psi^{(0)} (\delta z)_{st} -   \sum_{i=1}^{k}  \psi_s^{(i)} (\delta \tilde{z})_{s-r_i,
  t-r_i}.
\end{eqnarray*}

\noindent
{\it (i)} We first have to show that $\hro^{1}, \hro^{2} \in
\cac_{2}^{2\kappa}([a,b]; \R^{l})$.
For the second remainder term Taylor's formula yields
\begin{eqnarray*}
|\hro^2_{st}|&\le&
\frac{1}{2}\|     \vp''\|_{\infty} \left(  | (\delta z)_{st}|^2 + \sum_{i=1}^{k} | (\delta  \tilde{  z})_{s-r_i,t-r_i}|^2 \right),
\end{eqnarray*}
and hence clearly, thanks to some straightforward bounds in the spaces $\cq$, we have
\begin{equation}\label{est_r_1}
\frac{|\hro^2_{st}|} { \, \, |t-s|^{2 \kappa}}\le
\frac12\|   \vp''\|_{\infty}
\lp  \cn^2[z; \cq_{\kappa, \alpha}([a,b]; \R^{n}) ] + \sum_{i=1}^{k}
\cn^2[\tilde{z} ;      \cq_{\kappa, \alpha}([a-r_i,b-r_i]; \R^{n}) ] \rp.
\end{equation}
The first term can also be bounded easily: it can be checked  that
\begin{equation}\label{est_r_2}
\frac{|\hro_{st}^{1}| } { \, \, |t-s|^{2 \kappa}}\le
\| \varphi' \|_{\infty} \, \lp\cn \lc \rho ; \cac_{2}^{2 \kappa}([a,b]; \R^{n})\rc
+  \sum_{i=1}^{k}\cn  \lc  \tilde\rho , \cac_{2}^{2 \kappa}( [a-r_i,b-r_i]; \R^{n}) \rc \rp
\end{equation}
Putting together the last two inequalities, we have shown that decomposition  (\ref{decompo:ttz}) holds, that is
$$
\lp \delta \hz \rp_{st}
=      \psi^{(0)}  \zeta_s (\delta x)_{s, t}  +   \sum_{i=1}^{k} \psi_s^{(i)} \tilde{\zeta}_{s-r_i}^{(i)}  (\delta x)_{s-r_i, t-r_i}      +\hro_{st}$$
with $\hro_{st} =  \hro^1_{st} + \hro^{2}_{st} \in \cac_{2}^{2\kappa}([a,b]; \R^{d}).$

\medskip
\noindent
{\it (ii)} Now we have to consider the ``density'' functions
$$ \hat{\zeta}_{s}= \psi^{(0)}_s \zeta_s , \quad
\hat{\zeta}_{s}^{(i)}= \psi^{(i)}_s \tilde{\zeta}_{s-r_i},
\qquad s \in [a,b] .
$$
Clearly $ \hat{\zeta}, \hat{\zeta}^{(i)}$ are  bounded on $[a,b]$, because
the functions  $\psi^{(i)}$ are bounded (due to the boundedness of $
\varphi'$)  and because $\zeta$, $\tilde{\zeta}^{(i)}$ are also bounded.  In particular, it holds
\begin{eqnarray} \label{est_r_3}
  \sup_{ s \in [a,b]   } |\hat{\zeta}_{s}| \leq \| \varphi'
  \|_{\infty}  \sup_{ s \in [a,b] } |\zeta_{s}|, \qquad  \qquad \sup_{ s \in [a,b]   } |\hat{\zeta}_{s}^{(i)}| \leq \| \varphi' \|_{\infty}  \sup_{ s \in [a,b] } |\tilde{\zeta}_{s-r_i}|
\end{eqnarray} for $i=1, \ldots, k$.
Moreover, for $i=1, \ldots, k$, we have
\begin{align}  \label{est_r_4}  \nonumber
&  |  \hat{\zeta}_{s_1}^{(i)}- \hat{\zeta}_{s_2}^{(i)} | \nonumber
 \\ & \quad  \leq     | (\psi_{s_1}^{(i)} - \psi_{s_2}^{(i)})  \tilde{\zeta}_{s_1-r_i}|
 +   |(\tilde{\zeta}_{s_1-r_i}-   \tilde{\zeta}_{s_2-r_i} )
 \psi_{s_2}^{(i)}) |  \nonumber
 \\ \nonumber & \quad
   \leq  \| \varphi'' \|_{\infty}   |z_{s_1}- z_{s_2}|   \sup_{s \in [a,b]}| \tilde{\zeta}_{s-r_i}|  +
 \| \varphi'' \|_{\infty}   \sum_{j=1}^{k}|
 \tilde{z}_{s_1-r_j}- \tilde{z}_{s_2-r_j}|   \sup_{s \in [a,b]}|
 \tilde{\zeta}_{s-r_i}|      \\ & \qquad   + \nonumber
  \| \psi^{(i)} \|_{\infty}
 |\tilde{\zeta}_{s_1-r_i}- \tilde{\zeta}_{s_2-r_i}|
\\ &  \quad \leq  \|  \varphi'' \|_{\infty}   \ \cn[z ; \cac_{1}^{\kappa}([a,b]; \R^{n})] \,
\sup_{s \in [a,b]}
| \tilde{\zeta}_{s-r_i}|  \,  |s_2-s_1|^{\kappa}   \\ & \qquad  \nonumber    +
  \| \varphi'' \|_{\infty} \sum_{j=1}^{k}    \cn[ \tilde{z} ; \cac_{1}^{\kappa}([a-r_j,b-r_j]; \R^{n}) ] \, \sup_{s \in [a,b]}
| \tilde{\zeta}_{s-r_i}|  \,   |s_2-s_1|^{\kappa}     \\ & \qquad
+       \| \psi^{(i)}
\|_{\infty} \,    \cn   [
\tilde{\zeta} ; \cac_{1}^{\kappa}  (    [a-r_i,b-r_i] ; \R^{n}
) ] \, |s_2-s_1|^{\kappa}.  \nonumber
\end{align}
 Similarly, we obtain
\begin{eqnarray} \nonumber
  | \hat{\zeta}_{s_1}- \hat{\zeta}_{s_2} | & \leq & \| \varphi'' \|_{\infty}   \ \cn[z ; \cac_{1}^{\kappa}([a,b]; \R^{n})] \, \sup_{s \in [a,b]}  \nonumber
| \zeta_{s}| \,  |s_2-s_1|^{\kappa}   \\ && \nonumber  \,  +
  \| \varphi'' \|_{\infty} \sum_{j=1}^{k}    \cn[ \tilde{z} ; \cac_{1}^{\kappa}([a-r_j,b-r_j]; \R^{n}) ] \, \sup_{s \in [a,b]}
| \zeta_{s}|  \,   |s_2-s_1|^{\kappa}     \\ &&  \,
+       \| \psi^{(i)}
\|_{\infty} \,  \cn     [
\zeta ; \cac_{1}^{\kappa}  (    [a,b] ; \R^{n}       ) ] \,|s_2-s_1|^{\kappa}. \label{est_r_5}
\end{eqnarray}
Hence, the densities satisfy the conditions of Definition \ref{def:dcp}.

\medskip
\noindent
{\it (iii)} Finally, combining the estimates  (\ref{est_r_1}),
(\ref{est_r_2}), (\ref{est_r_3}) and (\ref{est_r_4}) yields the estimate  (\ref{bnd:t-phi-z}), which ends the proof.
\end{proof}

\vspace{0.3cm}

We  thus have  proved that the map $T_{\varphi}$ is quadratically bounded
in  $z$ and $\tilde{z}$. Moreover, for fixed  $\tilde{z}$ the map $T_{\varphi}(\cdot,
\tilde{z}): \cq_{\ka,\alpha
  }([a;b];
  \R^{d}) \rightarrow  \cd_{\ka,\hat{\al}
  }([a;b];
  \R^{d})] $
 is locally Lipschitz continuous:

\bigskip

\begin{proposition}\label{compo:ccp-loc-lin}
 Let the notation of Proposition \ref{compo:ccp-phi} prevail. Let $0 \leq a \leq b \leq T$,
let $z^{(1)}, z^{(2)} \in \cq_{\kappa, \alpha}([a,b];\R^n)$ and let
$\tilde{z} \in \cq_{\kappa, \tilde{\alpha}} ([a-r_k, b-r_1]; \R^n)$.
Then,
\begin{align}  \label{contract_2}  &\cn[T_{\varphi}(z^{(1)},
  \tilde{z}) - T_{\varphi}(z^{(2)}, \tilde{z}) ;\cd_{\ka,0
  }([a;b];
  \R^{d})]    \\ & \qquad \qquad \qquad \qquad \le c_{\varphi,T} \,
  \big(1+ C(z^{(1)},z^{(2)}, \tilde{z})\big)^2 \,
  \cn [
z^{(1)}-z^{(2)};\cq_{\ka,\alpha}([a,b];\R^{n})],  \nonumber
\end{align}
where
\begin{eqnarray}
&&C(z^{(1)},z^{(2)}, \tilde{z})=  \cn[ \tilde{z}
;\cq_{\ka,\tilde{\alpha}}([a-r_k,b-r_1];\R^{n})] \nonumber  \\
&&\hskip3.5cm+ \cn[z^{(1)} ;\cq_{\ka,\alpha}([a,b];\R^{n})] + \cn[z^{(2)}
;\cq_{\ka,\alpha}([a,b];\R^{n})]
\label{23bis}
 \end{eqnarray}
and the   constant $c_{\varphi,   T }$ depends only  on $\varphi$ and
$T$.
\end{proposition}

\begin{proof}
Denote $\hz^{(j)}= T_{\sigma} (z^{(j)}, \tilde{z})$ for $j=1,2$.
By Proposition \ref{compo:ccp-phi} we have
\begin{equation*}
\lp \delta \hz^{(j)} \rp_{st}
=  \hat \zeta_s^{(j)} \,  (\delta x)_{st} +
\sum_{i=1}^{k} \hat\zeta_s^{(i,j)}\,  (\delta x)_{s-r_i,t-r_i} +\hro_{st}^{(j)}, \qquad s,t \in [a,b] \end{equation*}
 with
$$ \hat \zeta_s^{(j)} =   \psi_s^{(0,j)} \zeta_{s}^{(j)}, \qquad \hat
\zeta_s^{(i,j)} = \psi_s^{(i,j)}\tilde{\zeta}_{s-r_i}, \qquad s \in
[a,b], $$
where
\begin{eqnarray*}
 \psi_s^{(i,j)} &=&    \left(  \frac{ \partial \vp }{
      \partial x_{1,i}}   (z_s^{(j)}, \mathfrak{s}(\tilde{z})_{s})
 , \ldots  ,  \frac{ \partial \vp }{ \partial x_{n,i}}
  (z_s^{(j)}, \mathfrak{s}(\tilde{z})_{s}) \right)   , \qquad s\in [a,b],
\end{eqnarray*}
 for $i=0, \ldots, k$, $j=1,2$.
Furthermore, it holds $\hat{\rho}_{st}^{(j)}= \hro_{st}^{(1,j)} +
\hro_{st}^{(2,j)},$ where
\begin{align*}
\hro_{st}^{(1,j)}&= \psi_s^{(0,j)} \rho^{(j)}_{st}
+ \sum_{i=1}^{k}  \psi_s^{(i,j)} \tilde{\rho}_{s-r_i,t-r_i} , \nonumber\\
\hro_{st}^{(2,j)}&= \sigma (  z_{t-r_0}^{(j)},  \tilde{z}_{t-r_1}, \ldots, \tilde{z}_{t-r_k}) -\sigma (  z^{(j)}_{s-r_0}, \tilde{z}_{s-r_1},
\ldots, \tilde{z}_{s-r_k})  \\ &  \qquad \qquad  \qquad \qquad\qquad \qquad \qquad \qquad  - \psi^{(0,j)}_s (\delta z^{(j)})_{st} -   \sum_{i=1}^{k}  \psi_s^{(i,j)} (\delta \tilde{z})_{s-r_i,
  t-r_i}.
\end{align*}
Thus, we obtain for $\hat{z}=\hz^{(1)}-\hz^{(2)}$ the decomposition
$$(\der \hat{z})_{st}
=
\sum_{i=0}^{k}\, \hat{\zeta}_s^{(i)}\, (\delta x)_{s-r_i,t-r_i}
+ \hat{\rho}_{st}
$$
with $\hat{\zeta}^{(0)}_s= \psi_s^{(0,1)} \zeta_{s}^{(1)} - \psi_s^{(0,2)}
\zeta_{s}^{(2)}$, the paths $\hat{\zeta}^{(i)}$ are defined by $\hat{\zeta}^{(i)}_s=(
\psi_s^{(i,1)}-\psi_s^{(i,2)})\tilde{\zeta}_{s-r_i}$ for  $ i=1, \ldots, k$, and
$\hat{\rho}_{st}= \hat{\rho}_{st}^{(1)} - \hat{\rho}_{st}^{(2)}$.

\medskip

In the following we will denote constants (which depend only on $T$
and $\varphi$) by $c$, regardless of their value.  For
convenience, we will also use the short notations  $\cn[
\tilde{z}]$, $\cn[z^{(1)}]$, $\cn[z^{(2)}]$ and  $\cn[z^{(1)}-z^{(2)}]$
instead of the corresponding quantities in (\ref{contract_2})-(\ref{23bis}).

\medskip
\noindent
{\it (i)}  We first control the supremum of the density functions
$\zeta^{(i)}$, $i=0, \ldots, k$.
For $i=0$, we can write
\begin{align*}
\hat{\zeta}^{(0)}_s= \psi_s^{(0,1)} ( \zeta_{s}^{(1)} -  \zeta_{s}^{(2)}) +
(\psi_s^{(0,1)}- \psi_s^{(0,2)}) \zeta^{(2)}_s
\end{align*}
and thus it follows
\begin{align} \label{lip_con_T_-1} \nonumber
| \hat{\zeta}_s^{(0)} | & \leq
\|\varphi' \|_{\infty} |\zeta_s^{(1)}- \zeta_s^{(2)}|
| \zeta^{(2)}_s|  +
\|\varphi'' \|_{\infty} |z_s^{(1)}- z_s^{(2)}| \\ & \leq  c \left(1+ \cn[
  z^{(2)}  ] \right) \, \cn[z^{(1)}-z^{(2)}]
\end{align}
Similarly, we get
\begin{align} \label{lip_con_T_0}
| \hat{\zeta}_s^{(i)} |  \leq  c   \,  \cn[
\tilde{z}]   \,  \cn[z^{(1)}-z^{(2)}].
\end{align}

\medskip
\noindent
{\it (ii)} Now, consider the increments of the density functions.
 Here, the key is to expand the expression $\psi_s^{(i,1)}-\psi_s^{(i,2)}$ for $i=
 0, \ldots, k$. For this define
$$ u_s(r)= r(z^{(1)}_s - z^{(2)}_s) + z^{(2)}_s, \qquad r \in[0,1],
\quad s \in [a,b].$$
We have
\begin{eqnarray*}
   \frac{ \partial \vp }{
      \partial x_{l,i}}  (z_s^{(1)}, \mathfrak{s}(\tilde{z})_{s}) - \frac{ \partial \vp }{
      \partial x_{l,i}}  (z_s^{(2)}, \mathfrak{s}(\tilde{z})_{s})&
    =&  \frac{ \partial \vp }{
      \partial x_{l,i}}  (u_s(1), \mathfrak{s}(\tilde{z})_{s}) - \frac{ \partial \vp }{
      \partial x_{l,i}}  (u_s(0), \mathfrak{s}(\tilde{z})_{s}) \\
&     = &     \theta^{(l,i)}_s     (z^{(1)}_s - z^{(2)}_s)  ,
\end{eqnarray*}
where
$$   \theta^{(l,i)}_s   =  \int_{0}^{1} \left(  \frac{ \partial^2 \vp }{
      \partial x_{1,0}\partial x_{l,i}}  (u_s(r), \mathfrak{s}(\tilde{z})_{s}) , \ldots , \frac{ \partial^2 \vp }{
      \partial x_{n,0}\partial x_{l,i}}   (u_s(r), \mathfrak{s}(\tilde{z})_{s}) \right) \,
  dr.$$
Hence it follows
\begin{align} \label{decomp_psi}       \psi_s^{(i,1)}-\psi_s^{(i,2)} =  \left(      \theta^{(1,i)}_s
  (z^{(1)}_s - z^{(2)}_s)  , \ldots,   \theta^{(n,i)}_s
  (z^{(1)}_s - z^{(2)}_s)  \right). \end{align}
Note that $\theta^{(l,i)}$ is clearly bounded and, under the assumption
$\varphi \in C_b^3$, it moreover satisfies:
\begin{align}    |\theta^{(l,i)}_t -                \theta^{(l,i)}_s
  | & \leq   c \, \left(
  \cn[z^{(1)}]     +    \cn[z^{(2)}]     +  \cn[ \tilde{z}]  \right)  \, |t-s|^{\kappa}
.   \label{ich_lauf_noch_amok} \end{align}

\medskip
For $i=0$ we can now write
\begin{align*}
\hat{\zeta}_t^{(0)}-\hat{\zeta}^{(0)}_s &=
\left(\psi_t^{(0,1)} - \psi_s^{(0,1)} \right) ( \zeta_{s}^{(1)} -  \zeta_{s}^{(2)}) +
\psi_t^{(0,1)} \left( (\zeta^{(1)}_t -\zeta_t^{(2)}) - (\zeta^{(1)}_s
  -\zeta_s^{(2)}) \right) \\ & \quad + \left(\psi_s^{(0,1)} -
  \psi_s^{(0,2)} \right) ( \zeta_{t}^{(2)} -  \zeta_{s}^{(2)}) + \zeta_t^{(2)} \left(
(\psi^{(0,1)}_t -\psi_t^{(0,2)}) - (\psi^{(0,1)}_s
  -\psi_s^{(0,2)}) \right).
\end{align*}
It follows
\begin{align} \label{dens_est_a} \nonumber
| \hat{\zeta}_t^{(0)}-\hat{\zeta}^{(0)}_s |  &\leq   c \left( \cn[z^{(1)}]  + \cn[
\tilde{z}]   \right)  \, |t-s|^{\kappa}  \,
\cn[z^{(1)}-z^{(2)}]+  c  \,
\cn[z^{(1)}-z^{(2)}] \, |t-s|^{\kappa}  \\   & \qquad  \nonumber
+ c  \, \cn[z^{(1)}-z^{(2)}] \,   \cn[z^{(2)}] \, |t-s|^{\kappa}
\\  & \qquad  +\cn[z^{(2)}] \left| (\psi^{(0,1)}_t -\psi_t^{(0,2)}) - (\psi^{(0,1)}_s
  -\psi_s^{(0,2)}) \right|  .
\end{align}
Using  (\ref{decomp_psi}) and (\ref{ich_lauf_noch_amok})
we obtain \begin{align}  \label{dens_est_b}
& \left|(\psi^{(0,1)}_t -\psi_t^{(0,2)}) - (\psi^{(0,1)}_s
  -\psi_s^{(0,2)})\right| \\ & \qquad  \leq c     \left( 1+ \cn[z^{(1)}] + \cn[z^{(2)}]  + \cn[
\tilde{z}]  \right)  \, \cn[ z^{(1)}-z^{(2)}] \, |t-s|^{\kappa}.
\nonumber
\end{align}
Combining   (\ref{dens_est_a}) and   (\ref{dens_est_b}) yields
\begin{align}\label{lip_con_T_1}
| \hat{\zeta}_t^{(0)}-\hat{\zeta}^{(0)}_s | & \leq   c \left( 1 + \cn[z^{(1)}] +
  \cn[z^{(2)}] + \cn[
\tilde{z}]  \right)^2  \, \cn[z^{(1)}-z^{(2)}] \,  |t-s|^{\kappa} .
\end{align}
By similar calculations we also have
\begin{align} \label{lip_con_T_2}
| \hat{\zeta}_t^{(i)}-\hat{\zeta}^{(i)}_s | & \leq   c \left( 1 + \cn[z^{(1)}] +
  \cn[z^{(2)}] + \cn[
\tilde{z}]   \right)^2  \, \cn[z^{(1)}-z^{(2)}] \,  |t-s|^{\kappa}
\end{align}
for $i=1, \ldots, k$.

\medskip

\medskip
\noindent
{\it (iii)} Now, we have to control the remainder term $\hat{\rho}$. For this we
decompose $\rho$ as
$$ \hat{\rho}_{st} =\rho_{st}^{(1)} + \rho^{(2)}_{st}, $$ where
\begin{align*}
\rho_{st}^{(1)}&=
\psi_s^{(0,1)}\rho^{(1)}_{st}-\psi_s^{(0,2)}\rho^{(2)}_{st}
+\sum_{i=1}^{k} \left(\psi_s^{(i,1)} -\psi_s^{(i,2)}\right) \tilde{\rho}_{s-r_i,t-r_i} ,  \\
\rho_{st}^{(2)}&= \left( \varphi (  z_{t-r_0}^{(1)},  \tilde{z}_{t-r_1}, \ldots, \tilde{z}_{t-r_k}) -\varphi (  z^{(1)}_{s-r_0}, \tilde{z}_{s-r_1},
\ldots, \tilde{z}_{s-r_k}) \right) \\& \qquad \quad  - \left( \varphi (  z_{t-r_0}^{(2)},  \tilde{z}_{t-r_1}, \ldots, \tilde{z}_{t-r_k}) -\varphi (  z^{(2)}_{s-r_0}, \tilde{z}_{s-r_1},
\ldots, \tilde{z}_{s-r_k}) \right)
  \\ &  \qquad \quad   - \left(  \psi^{(0,1)}_s (\delta
    z^{(1)})_{st} -   \psi^{(0,2)}_s (\delta
    z^{(2)})_{st}      \right) -   \sum_{i=1}^{k} \left(
    \psi_s^{(i,1)}  - \psi_s^{(i,2)}    \right) (\delta \tilde{z})_{s-r_i,
  t-r_i}.
\end{align*}

We consider first $\rho^{(1)}$: for this term, some straightforward calculations yield
\begin{align} \label{lip_con_T_3}
|\rho_{st}^{(1)}| \leq   c ( 1 + \cn[z^{(2)}]+ \cn[\tilde{z}]  )  \, \cn[z^{(1)}-z^{(2)}] \,
|t-s|^{2 \kappa}
 . \end{align}

Now consider $\rho^{(2)}$.  The mean value theorem yields
\begin{align*}
\rho_{st}^{(2)}&=
  \left( \bar{\psi}^{(0,1)}_s  -
    \psi^{(0,1)}_s  \right) (\delta
    z^{(1)})_{st} - \left(  \bar{\psi}^{(0,2)}_s  -\psi^{(0,2)}_s
    \right)  (\delta
    z^{(2)})_{st}       \\ & \qquad  \qquad  +   \sum_{i=1}^{k} \left(
       \left( \bar{\psi}^{(i,1)}_s -   \bar{\psi}^{(i,2)}_s
       \right)- \left( \psi_s^{(i,1)}  - \psi_s^{(i,2)}
       \right)\right) (\delta \tilde{z})_{s-r_i,t-r_i}
\\ & = \left( \bar{\psi}^{(0,1)}_s  -
    \psi^{(0,1)}_s  \right) (\delta
    (z^{(1)} -z^{(2)}) )_{st}    \\
       &  \qquad  \qquad
    + \left( \left( \bar{\psi}^{(0,1)}_s -   \bar{\psi}^{(0,2)}_s
       \right)- \left( \psi_s^{(0,1)}  - \psi_s^{(0,2)}
       \right) \right) (\delta z^{(2)})_{st}       \\
       &   \qquad \qquad+   \sum_{i=1}^{k} \left(
       \left( \bar{\psi}^{(i,1)}_s -   \bar{\psi}^{(i,2)}_s
       \right)- \left( \psi_s^{(i,1)}  - \psi_s^{(i,2)}
       \right)\right) (\delta \tilde{z})_{s-r_i,t-r_i} \\ & \triangleq Q_1+Q_2+Q_3,
\end{align*}
with
\begin{align*}
    \bar{ \psi}_{s}^{(i,j)}   &=  \int_{0}^{1} \left(  \frac{ \partial \vp }{
      \partial x_{1,i}}  ( v^{(j)}_s(r)) , \ldots ,  \frac{ \partial \vp }{
      \partial x_{n,i}}  (v^{(j)}_s(r)) \right)
  dr  , \\
   v^{(j)}_s(r) &= \left( z^{(j)}_{s} + r(z^{(j)}_t - z^{(j)}_s ),
  \tilde{z}_{s-r_1} + r( \tilde{z}_{t-r_1} - \tilde{z}_{s-r_1} ),
  \ldots,  \tilde{z}_{s-r_k} + r( \tilde{z}_{t-r_k} - \tilde{z}_{s-r_k}
  ) \right).
\end{align*}

We shall now bound $Q_1,Q_2$ and $Q_3$ separately: it is readily checked that
$$ |\bar{ \psi}_{s}^{(i,j)}-   \psi_{s}^{(i,j)}| \leq c \left( 1+
  \cn[z^{(1)}] + \cn[z^{(2)}] + \cn[\tilde{z}]  \right) \,
|t-s|^{\kappa},$$
and thus we obtain
\begin{equation} \label{loc_lin_T_4}
Q_1 \leq      c \left( 1+
  \cn[z^{(1)}] + \cn[z^{(2)}] + \cn[\tilde{z}]  \right) \, \cn[z^{(1)}-z^{(2)} ] \,
|t-s|^{2\kappa}.
\end{equation}

In order to estimate $Q_2$ and $Q_3$,
recall that by  (\ref{decomp_psi}) in part {\it(ii)} we have
\begin{align}    \label{decomp_psi_iii}
\psi_s^{(i,1)}-\psi_s^{(i,2)} =  \left(      \theta^{(1,i)}_s
  (z^{(1)}_s - z^{(2)}_s)  , \ldots,   \theta^{(n,i)}_s
  (z^{(1)}_s - z^{(2)}_s)  \right),
\end{align}
where
\begin{align*}
   \theta^{(l,i)}_s   &=  \int_{0}^{1} \left(  \frac{ \partial^2 \vp }{
      \partial x_{1,0}\partial x_{l,i}}    (u_s(r'), \mathfrak{s}(\tilde{z})_{s}) , \ldots ,
      \frac{ \partial^2 \vp }{
      \partial x_{n,0}\partial x_{l,i}}   (u_s(r'), \mathfrak{s}(\tilde{z})_{s}) \right)
  dr' ,  \\
  u_s(r')&= z_{s}^{(1)} + r' (z_s^{(2)}-z_s^{(1)}) .
\end{align*}
Similarly, we also obtain that
\begin{align} \label{decomp_psi_2}       \bar{\psi}_s^{(i,1)}-\bar{\psi}_s^{(i,2)} =  \left(     \bar{\theta}^{(1,i)}_s
  (z^{(1)}_s - z^{(2)}_s)  , \ldots,   \bar{\theta}^{(n,i)}_s
  (z^{(1)}_s - z^{(2)}_s)  \right) \end{align}
with
\begin{eqnarray*}
 \bar{\theta}^{(l,i)}_s   &=&  \int_{0}^{1} \int_{0}^{1}\left(  \frac{ \partial^2 \vp }{
      \partial x_{1,0}\partial x_{l,i}}   (\bar{u}_s(r,r')) , \ldots ,
      \frac{ \partial^2 \vp }{
      \partial x_{n,0}\partial x_{l,i}}   (\bar{u}_s(r,r')) \right)
  dr \, dr'  \\
  \bar{u}_s(r,r')&= &   v_s^{(1)}(r)+
r' \left(v_s^{(2)}(r)- v_s^{(1)}(r)        \right).
\end{eqnarray*}
Now, using (\ref{decomp_psi_iii}) and (\ref{decomp_psi_2}) we can   write
\begin{align*}
 & \left( \bar{\psi}^{(i,1)}_s -   \bar{\psi}^{(i,2)}_s
       \right)- \left( \psi_s^{(i,1)}  - \psi_s^{(i,2)}
       \right)  \\ & \qquad \qquad   = \left( (\bar{\theta}_s^{(1,i)}-
         \theta_s^{(1,i)}) (z^{(1)}-z^{(2)})  , \ldots ,     (\bar{\theta}_s^{(n,i)}-
         \theta_s^{(n,i)}) (z^{(1)}-z^{(2)}  \right)
\end{align*} for any
$i=0,\ldots,k$.
Since moreover
\begin{align*}
&\bar{u}_s(r,r')- (u_s(r'), \mathfrak{s}(\tilde{z})_s) \\ & \quad= r\left( ( z^{(1)}_t-z^{(1)}_s) + r'
    (z_{t}^{(2)}-z_{s}^{(2)}- (z_{t}^{(1)}-z_{s}^{(1)}  )),
    \tilde{z}_{t-r_1}- \tilde{z}_{s-r_1}, \ldots,     \tilde{z}_{t-r_k}- \tilde{z}_{s-r_k}       \right),
\end{align*}
another Taylor expansion yields
$$  |\bar{\theta}_s^{(l,i)}-
         \theta_s^{(l,i)})| \leq   c \, \left( \cn[z^{(1)}]    +  \cn[z^{(2)}]
         +  \cn[ \tilde{z}]  \right) \,  |t-s|^{\kappa} .    $$
Hence, we obtain
\begin{align}\label{loc_lin_T_5}
& \left| \left( \bar{\psi}^{(i,1)}_s -   \bar{\psi}^{(i,2)}_s
       \right)- \left( \psi_s^{(i,1)}  - \psi_s^{(i,2)}
       \right) \right|  \nonumber  \\ & \qquad \qquad  \leq c \, \left( \cn[z^{(1)}]    +  \cn[z^{(2)}]
         +  \cn[ \tilde{z}] \right) \,  \cn[z^{(1)}-z^{(2)}]  \, |t-s|^{\kappa},
\end{align}
from which suitable bounds for $Q_2$ and $Q_3$ are easily deduced.
Thus it follows by  (\ref{loc_lin_T_4}) and (\ref{loc_lin_T_5}) that
 \begin{align*}
|\rho_{st}^{(2)}|&  \leq c \, \left(1+ \cn[z^{(1)}]    +  \cn[z^{(2)}]
         +  \cn[ \tilde{z}] \right)^2 \,  \cn[z^{(1)}-z^{(2)}]  \,
       |t-s|^{2\kappa}. \end{align*}
Combining this estimate with (\ref{lip_con_T_3}) we finally have
\begin{align} \label{loc_lin_T_6}
|\rho_{st}|&  \leq c \, \left(1+ \cn[z^{(1)}]    +  \cn[z^{(2)}]
         +  \cn[ \tilde{z}] \right)^2 \,  \cn[z^{(1)}-z^{(2)}]  \,
       |t-s|^{2\kappa}. \end{align}

\medskip
\noindent
{\it (iv)}  The assertion follows now from  (\ref{lip_con_T_-1}),
(\ref{lip_con_T_0}), (\ref{lip_con_T_1}),
(\ref{lip_con_T_2}) and  (\ref{loc_lin_T_6}).

\end{proof}


\subsection{Integration of delayed controlled paths (DCP)}

The aim of this section is to define the integral $\cj(m^{*} dx)$,
where $m$ is  a delayed controlled path
$m\in\cd_{\ka,\al}([a,b];\R^{d})$.
Here we denote by  $A^{*}$  the
transposition of a vector or matrix $A$ and by
 $A_1 \cdot A_2 $ the inner product of two vectors or
two matrices $A_1$ and $A_2$. We will also write $\cq_{\kappa,\alpha}$ (resp.
$\mathcal{D}_{\kappa,\alpha}$) instead
of $\cq_{\kappa,\alpha}([a,b];V)$ (resp. $\mathcal{D}_{\kappa, \alpha}([a,b];V)$)
if there is no risk of confusion about $[a,b]$ and $V$.

Note that if the increments of $m$ can be expressed like in (\ref{decompo:dcp}),
 $m^{*}$ admits the decomposition
\begin{equation}\label{dcp:delta-m}
(\delta m^{*})_{st}
= \sum_{i=0}^{k} (\delta x)^{*}_{s-r_i,t-r_i} {\zeta^{(i) *}_s} + \rho^{*}_{st},
\end{equation}
where
$\rho^{*}\in\cac_2^{2\ka}([a,b];\R^{1,d})$  and the densities $\zeta^{(i)}$,
$i=0, \ldots, k$ satisfy
the conditions of Definition \ref{def:dcp}.

 To illustrate the structure of the integral of a DCP, we first assume
 that  the paths $x,\zeta^{(i)}$ and $\rho$ are smooth, and  we express
$\cj(m^{*} dx)$ in terms of the operators $\delta$
and $\laa$. In this case, $\cj(m^{*} dx)$ is well defined, and
we have
$$
\ist m_u^{*} dx_u = m_s^{*}(x_t-x_s) + \ist(m_u^{*}-m_s^{*}) dx_u
$$
for $a\leq s\le t \leq b$, or in other words
\begin{equation}\label{exp1:imdx}
\cj(m^{*} \,dx)= m^{*}\,\der x + \cj(\der m ^{*} \, dx).
\end{equation}
Now consider the term $\cj(\der m^{*} \, dx)$:   Using the
 decomposition (\ref{dcp:delta-m}) we
obtain
\begin{equation}\label{exp1:idmdx}
\cj(\der m^{*} \, dx)
=
\ist \lp \sum_{i=0}^{k} \, (\der x)_{s-r_i,u-r_i}^* \zeta_{s}^{(i) *}+\rho_{su}^{*} \rp dx_u
=
A_{st}+\cj_{st}(\rho^{*}\, dx)
\end{equation}
with
$$
A_{st}=\sum_{i=0}^{k}
\ist  (\der x)_{s-r_i,u-r_i}^* \zeta_{s}^{(i)*} dx_u.
$$
 Since, for the moment, we are dealing with smooth paths, the density $\zeta^{(i)}$ can be taken
out of the integral above, and we have
$$
A_{st}=\sum_{i=0}^{k}   \zeta_{s}^{(i)} \cdot  \xdst(-r_i),
$$
with the $d \times d$ matrix $\xdst(v)$ defined by
$$ \xdst(v)= \left( \ist \left(\int_{s+ v }^{u+v} dx_w \right) dx_u^{(1)}
  \, , \,
\ldots \, , \,  \ist \left(\int_{s+ v }^{u+v} dx_w \right) dx_u^{(d)} \right),
\qquad 0 \leq s \leq t \leq T
$$ for $v \in \{-r_k,\ldots,-r_0\}$. Indeed, we can write
\begin{eqnarray*}
\int_s^t (\delta x)^*_{s-r_i,u-r_i}\zeta^{(i)*}_s dx_u &= &\int_s^t \zeta^{(i)}_s\cdot
[(\delta x)_{s-r_i,u-r_i}\otimes dx_u] \\
&=& \zeta^{(i)}_s\cdot\int_s^t (\delta x)_{s-r_i,u-r_i}\otimes
dx_u=\zeta^{(i)}_s\cdot\xdst(-r_i).
\end{eqnarray*}
Inserting
the expression of $A_{st}$ into  (\ref{exp1:imdx}) and (\ref{exp1:idmdx})
we obtain
\begin{equation}\label{exp2:imdx}
\cj_{st}(m^{*}\, dx)
=
m_s^{*} (\der x)_{st} + \sum_{i=0}^{k}    \zeta_{s}^{(i)}  \cdot \xdst(-r_i)
+ \cj_{st}(\rho^{*}\, dx)
\end{equation}
for $ a\leq s \leq t \leq b$.
\vspace{0.3cm}

Let us  now consider  the L\'evy area term $\xdst(-r_i)$.  If $x$ is a smooth
path, it is readily checked that
$$
[\der\xd(-r_i)]_{sut}
=\bx_{st}^{\bf 2}(-r_i)-\bx_{su}^{\bf 2}(-r_i)-\bx_{ut}^{\bf 2}(-r_i)
=(\der x)_{s-r_i,u-r_i}\otimes(\der x)_{ut},
$$ for any $i=0,\ldots,k$.
This decomposition of $\der\xd(-r_i)$ into a product of increments is
the fundamental algebraic property we will use  to extend
the above integral to non-smooth paths.  Hence, we will need the
  following assumption:

\begin{hypothesis}\label{hyp:x}
The path $x$ is a $\R^d$-valued
$\ga$-H\"older continuous function with $\ga>\frac13$ and  admits a delayed L\'evy area,
i.e., for all $v\in\{-r_k,\ldots,-r_0\}$, there exists  a path
$\xd(v)\in\cac_2^{2\ga}([0,T];\R^{d,d})$,
which satisfies
\begin{equation}\label{cdf}
\der\xd(v)=\der x^v\otimes \der x,
\end{equation}
that is
$$
\lc (\der\xd(v))_{sut} \rc(i,j)
=
[\der x^{i}]_{s+v,u+v} [\der x^{j}]_{ut} \qquad  \textrm{for all} \qquad
s,u,t\in\ott, \quad i,j\in\{1,\ldots,d  \}.
$$ In the above formulae,
we have set $x^v$ for the shifted path $x^v_s=x_{s+v}$.
\end{hypothesis}

To finish the analysis
of the smooth case it remains to find a suitable expression for
$\cj(\rho^{*} \,dx)$. For this,  we write (\ref{exp2:imdx})
as
\begin{equation}\label{exp1:irhodx}
\cj_{st}(\rho^{*} \,dx)=
\cj_{st}(m^{*}\, dx)
-
m_s^{*} (\der x)_{st}
- \sum_{i=0}^{k} \zeta_s^{(i)}\cdot \xdst(-r_i)
\end{equation} and we apply  $\der$ to both  sides of the above equation.
For smooth paths $m$ and $x$
we have
$$
\der(\cj(m^{*}\, dx))=0,
\qquad \qquad
\der(m^{*}\,\der x)= - \der m^{*}\, \der x,
$$
by Proposition \ref{difrul}.
Hence, applying these relations to the right hand side of
(\ref{exp1:irhodx}), using the decomposition (\ref{dcp:delta-m}) and
again Proposition \ref{difrul},
we obtain
\begin{align*}
&[\der(\cj(\rho^{*}\, dx))]_{sut}\\
&=
(\der m^{*})_{su} (\der x)_{ut}
+\sum_{i=0}^{k} (\der\zeta^{(i)})_{su} \cdot \xdst(-r_i)
-\sum_{i=0}^{k} \zeta^{(i)}_{s}\cdot (\delta \xd(-r_i))_{sut} \\
&=\sum_{i=0}^{k}    (\der x)_{s-r_i,u-r_i}^{*}   \zeta_s^{(i) * }  (\der x)_{ut}
+\rho_{su}^{*}(\der x)_{ut}\\
&\hspace{5cm}
+\sum_{i=0}^{k} (\der\zeta^{(i)})_{su} \cdot \xdst(-r_i)
-\sum_{i=0}^{k} \zeta^{(i)}_{s}\cdot [    (\delta x)_{s-r_i,t-r_i}      \otimes     (\delta x)_{ut}   ]
\\
&=\rho_{su}^{*}(\der x)_{ut}+\sum_{i=0}^{k} (\der\zeta^{(i)})_{su} \cdot \xdst(-r_i).
\end{align*}
In summary, we have derived  the representation
$$
\der[\cj(\rho^{*}\, dx)]=\rho^{*}\, \der x + \sum_{i=0}^{k}
\der\zeta^{(i)} \cdot \xd(-r_i),
$$
for  two regular paths $m$ and $x$.

If $m, x, \zeta^{(i)}$, $i=0, \ldots, k$ and $\xd$ are smooth enough, we have
$\der[\cj(\rho^{*} \, dx)]\in\cz\cac_3^{1+}$ and thus  belongs to the domain
 of $\laa$   due  to
Proposition \ref{prop:Lambda}.  (Recall that $\der\der=0$.) Hence, it follows
$$
\cj(\rho^{*}\, dx)
=
\laa\lp \rho^{*}\, \der x  + \sum_{i=0}^{k}
\der\zeta^{(i)} \cdot \xd(-r_i) \rp,
$$ and
 inserting this identity into (\ref{exp2:imdx}), we end up with
\begin{equation}\label{exp3:imdx}
\cj(m^{*}\, dx)
=
m^{*} \der x
+ \sum_{i=0}^{k}
\zeta^{(i)} \cdot \xd(-r_i)
+ \laa\lp \rho^* \der x +  \sum_{i=0}^{k}
\der\zeta^{(i)} \cdot \xd(-r_i) \rp.
\end{equation}

\vspace{0.3cm}

The expression above can be generalised to the non-smooth case,
since $\cj(m^{*}\, dx)$ has been expressed only  in terms of increments of $m$
and $x$. Consequently, we will use (\ref{exp3:imdx}) as  the definition
for our extended integral.

\medskip

\begin{proposition}\label{intg:mdx}
For fixed $\frac13<\ka<\ga$,
let $x$ be a path satisfying Hypothesis \ref{hyp:x}. Furthermore,  let
$m\in\cd_{\ka,\hat{\alpha}}([a,b];\R^{d})$ such that the
increments of $m$ are given by (\ref{decompo:dcp}).
Define $z$ by $z_a=\alpha$ with $\alpha \in \R$  and
\begin{equation}\label{dcp:mdx}
(\der z)_{st}=
m_s^* ( \der x)_{st}
+ \sum_{i=0}^{k}
\zeta^{(i)}_s \cdot \xdst(-r_i)
+ \laa_{st}\lp \rho^* \der x +  \sum_{i=0}^{k}
\der\zeta^{(i)} \cdot \xd(-r_i) \rp
\end{equation}
for $a \leq s \leq t \leq b$.
Finally,
 set
 \begin{equation}
\cj(m^*\, dx) = \der z.
\end{equation}
Then:
\begin{enumerate}
\item
$\cj(m^*\, dx) $ coincides with the usual Riemann integral, whenever
$m$ and $x$ are smooth functions.
\item
$z$ is well-defined as an element of $\cq_{\ka,\alpha}([a,b];\R)$ with
decomposition $\der z= m^*\der x+\hro$,  where
$\hro\in\cac_2^{2\ka}([a,b]; \R)$
is given  by
$$
\hro=  \sum_{i=0}^{k}
\zeta^{(i)} \cdot \xd(-r_i)
+\laa\lp \rho^* \der x +  \sum_{i=0}^{k}
\der\zeta^{(i)} \cdot \xd(-r_i) \rp.
$$
\item
The semi-norm of $z$ can be estimated as
\begin{equation}\label{bnd:norm-imdx}
\cn[z;\cq_{\ka,\alpha}([a,b];\R)]\le \|m\|_\infty +
c_{ \it \textrm{int}} (b-a)^{\ga-\ka}\cn [m;\cd_{\ka,\hat{\alpha}}([a,b];\R^{d})]
\end{equation}
 where
$$
c_{ \it \textrm{int}} = c_{\kappa, \gamma, \varphi, T} \left(
\|x\|_\gamma + \sum_{i=0}^k \|\xd(-r_i)\|_{2\gamma}
\right)
$$
with the constant   $c_{\kappa, \gamma, \varphi, T}$ depending only on
$\kappa, \gamma, \varphi$ and $T$. Moreover,
\begin{equation}\label{bnd:norm-imdx-2}
\|z\|_\ga
\leq
 c_{ \it \textrm{int}} (b-a)^{\ga-\ka}\cn [m;\cd_{\ka,\hat{\alpha}}([a,b];\R^{d})].
\end{equation}

\item
It holds
\begin{equation}\label{rsums:imdx}
\cj_{st}(m^*\, dx)
=\lim_{|\Pi_{st}|\to 0}\sum_{i=0}^N
\lc m^*_{t_{i}}(\der x)_{t_{i}, t_{i+1}}
+ \sum_{j=0}^{k} \zeta^{(j)}_{t_{i}}  \cdot
{\bf x}^{\bf 2}_{t_{i}, t_{i+1}}(-r_j) \rc
\end{equation}
for any $a\le s<t\le b$,
where the limit is taken over all partitions
$\Pi_{st} = \{s=t_0,\dots,t_N=t\}$
of $[s,t]$, as the mesh of the partition goes to zero.
\end{enumerate}
\end{proposition}

\vspace{0.3cm}

\begin{proof}
(1) The first of our claims is a direct consequence of the derivation of
equation (\ref{exp3:imdx}).

\medskip
\noindent
(2)
Set
$
c_x=  \|x\|_\ga 
+ \sum_{i=0}^k \|\xd(-r_i)\|_{2\gamma}.
$
Now  we show that equation (\ref{dcp:mdx}) defines a classical controlled
path.  Actually, the term $m^*\, \der x$ is trivially of the desired form
for an element of $\cq_{\ka,\alpha}$.  So consider  the term
$  h^{(1)}_{st}=   \sum_{i=0}^{k}
\zeta^{(i)}_s \cdot \xdst(-r_i)    $ for $a \leq s \leq t \leq b$.
We have
\begin{align*}
|h^{(1)}_{st}|
&\le  \sum_{i=0}^{k}
\|\zeta^{(i)}\|_\infty  |\xdst(-r_i)|
\le  \left(  \sum_{i=0}^{k}
 \|\zeta^{(i)}\|_\infty \right )
 c_x |t-s|^{2\ga} \\
& \le  \left(  \sum_{i=0}^{k}
 \|\zeta^{(i)}\|_\infty \right )  c_x
(b-a)^{2(\ga-\ka)}|t-s|^{2\kappa}  .
\end{align*}
Thus
\begin{align*}
\| h^{(1)}\|_{2\ka} \leq   c_x
(b-a)^{2(\ga-\ka)}   \cn[m;\cd_{\ka, \hat{\alpha}}([a,b];\R^d)].
\end{align*}
The term
$$ h^{(2)} = \rho^* \der x +  \sum_{i=0}^{k}
\der\zeta^{(i)} \cdot {\bf x}^{\bf 2}(-r_i)  $$
satisfies $ \delta h^{(2)} =0$. Indeed, we can write
\begin{align*}
\der h^{(2)}&=\delta \rho^*\delta x + \sum_{i=0}^k \delta
\zeta^{(i)}\cdot \der {\bf x}^{\bf 2}(-r_i)
\end{align*}
by Proposition \ref{difrul} and because $\delta\delta=0$.
Applying (\ref{cdf}) to the right hand side of the above equation it follows that
\begin{multline*}
\der h^{(2)}=\delta \rho^*\delta x + \sum_{i=0}^k \delta
\zeta^{(i)}\cdot (\der x^{-r_i}\otimes\der x)
=\delta \rho^*\delta x + \sum_{i=0}^k \delta {x^{-r_i}}^* \delta \zeta^{(i)*}\der x\\
=\big(\delta\rho^* + \sum_{i=0}^k \delta {x^{-r_i}}^* \delta
\zeta^{(i)*}\big)\delta x.
\end{multline*}
However, due to Proposition \ref{difrul}, it holds
\begin{align*}
\der h^{(2)}=\big(\delta\rho^* + \sum_{i=0}^k \delta {x^{-r_i}}^* \delta
\zeta^{(i)*}\big)\delta x
=\delta\big(\rho^* + \sum_{i=0}^k \delta {x^{-r_i}}^*
\zeta^{(i)*}\big)\delta x.
\end{align*}
Since the increments of $m$ are given by (\ref{decompo:dcp}) we finally obtain that
\begin{align*}
 \der h^{(2)}=\der(\der m^*)\der x =0.
\end{align*}

Moreover, recalling the notation (\ref{eq:notation-norm-C3}), it holds
$$ \|\rho^* \der x\|_{2\ka,\ka} \leq   c_x (b-a)^{\gamma - \kappa}
\|\rho\|_{2\ka}
$$
and
$$
\|\sum_{i=0}^k\der\zeta^{(i)}\cdot {\bf x}^{\bf 2}(-r_i) \|_{\ka,2 \ka}
 \leq   c_x (b-a)^{2(\gamma - \kappa)}   \sum_{i=0}^{k}
\|\zeta^{(i)}\|_\ka. $$
Since $\gamma > \kappa > \frac13$ and $\delta h^{(2)}=0$, we have
$     h^{(2)} \in    \dom(\laa)   $ and
$$     \|h^{(2)} \|_{3 \kappa} \leq  c_x(1+ T^{\gamma - \kappa})
(b-a)^{\gamma - \kappa} \cn[m;\cd_{\ka, \hat{\alpha}}([a,b];\R^d)].$$
By Proposition \ref{prop:Lambda} it follows
$$  \|   \Lambda( h^{(2)}) \|_{ 3 \kappa } \leq \frac{1}{2^{3 \kappa} -2} \|
h^{(2)} \|_{3 \kappa} $$
and we finally obtain
\begin{align} \label{est_int_1}
      \|    h^{(1)} - \Lambda(h^{(2)}) \|_{2 \kappa}
 \leq   c_x \frac{2^{3 \kappa} -1 }{2^{3 \kappa} -2}(1+ T^{\gamma - \kappa})
(b-a)^{\ga-\ka} \cn[m;\cd_{\ka,\alpha}].
\end{align}

Thus we have proved that
$\hat{\rho}\in\cac_2^{2\ka}([a,b];\R)$
and hence that $z\in\cq_{\ka,\alpha}([a,b];\R)$.

\medskip
\noindent
(3)
Because of
$(\der z)_{st}=m^*_s ( \der x)_{st} + \hat{\rho}_{st}$
and $m \in \cd_{\ka, \hat{\alpha}}([a,b];\R^d)$
the estimates (\ref{bnd:norm-imdx}) and (\ref{bnd:norm-imdx-2})
 now follow from (\ref{est_int_1}).

\medskip
\noindent
(4)
By Proposition \ref{difrul} {\it(ii)} and the decomposition  (\ref{decompo:dcp}) we have that
\begin{align*}
\delta (m^* \delta x)_{sut}= -(\delta m^*)_{su} (\delta x)_{ut}  = - \rho^*_{su} (\delta x)_{ut} -
 \sum_{i=0}^{k}  (\delta
  x)_{s-r_i,u-r_i}^{*} \zeta^{(i)*}_{u}   (\delta x)_{ut}.
\end{align*}
Thus, applying again Proposition \ref{difrul} {\it(ii)}, and recalling Hypothesis \ref{hyp:x}
for the L\'evy area, we obtain that
$$
\der\lp m^*\, \der x+ \sum_{i=0}^{k} \zeta^{(i)} \cdot \xd(-r_i)  \rp
=
-\lc \rho^*\, \der x + \sum_{i=0}^{k} \delta \zeta^{(i)} \cdot \xd(-r_i) \rc.
$$
Hence, equation (\ref{dcp:mdx}) can also be written as
$$
\cj(m^*\, dx)=\lc \id-\laa\der \rc
\lp m^*\, \der x+ \sum_{i=0}^{k} \zeta^{(i)} \cdot \xd(-r_i) \rp,
$$
and a direct application of Corollary \ref{cor:integration} yields
(\ref{rsums:imdx}), which ends our proof.

\end{proof}

\medskip

Recall that the notation $A^*$ stands for the transpose of a matrix $A$.
Moreover, in the sequel, we will denote by $c_{norm}$ a constant, which
depends only on the chosen norm of $\R^{n,d}$.
Then, for a matrix-valued delayed controlled path $m\in\cd_{\ka,\hat{\alpha}}([a,b];\R^{n,d})$, the
integral $ \mathcal{J}(m \, dx )$ will be defined by
$$
 \mathcal{J}(m \, dx ) =  \left( \mathcal{J}(m^{(1)*} dx), \ldots, \mathcal{J}(m^{(n)*}dx) \right)^*,
$$
where $m^{(i)} \in\cd_{\ka,\hat{\alpha}}([a,b];\R^{d})$ for $i=1, \ldots, n$
and we have set $m=(m^{(1)}, \ldots,m^{(n)})^*$.
Then we have  by (\ref{bnd:norm-imdx}) that
\begin{align}\label{bnd:norm-imdx2}
& \cn[ \mathcal{J}(m \, dx );\cq_{\ka,\alpha}([a,b];\R^{n})]
\\ & \qquad \qquad  \qquad \le
   c_{norm} \left(\|m\|_\infty +   c_{\textrm{\it int}}
     (b-a)^{\ga-\ka}\cn   [m;\cd_{\ka,\hat{\alpha}}([a,b];\R^{n,d}) ]
   \right)  . \nonumber
\end{align}

For two paths  $m^ {(1)}, m^{(2)} \in\cd_{\ka,\hat{\alpha}}([a,b];\R^{n,d})$ we
obtain the following estimate for the difference of $z^{(1)}
=\mathcal{J}(m^{(1)} \, dx )$ and  $z^{(2)}
=\mathcal{J}(m^{(2)} \, dx )$:
As above, we have clearly
\begin{align*}
\cn[z^{(1)}-z^{(2)};\cq_{\ka,0}([a,b];\R^{n})] & \le
     c_{norm} \|m^{(1)} -m^{(2)}\|_\infty  \\ & \qquad +      c_{norm}\,
   c_{\textrm{\it int}}  (b-a)^{\ga-\ka}\cn \left
  [m^{(1)}-m^{(2)};\cd_{\ka,0}([a,b];\R^{n,d}) \right]  .
\end{align*}
However, since $m_a^{(1)} = m_a^{(2)}$ it follows
\begin{equation}\label{contract}
\cn[z^{(1)}-z^{(2)};\cq_{\ka,0}([a,b];\R^{n})]\le   2   c_{norm} \, c_{
  \textrm{\it int}}  (b-a)^{\ga-\ka}\cn \left
  [m^{(1)}-m^{(2)};\cd_{\ka,0}([a,b];\R^{n,d}) \right]  .
\end{equation}

\bigskip
\medskip

\section{Solution to the delay equation}

With the  preparations of the last section, we can now solve the equation
\begin{align}
\left\{\begin{array}{lr}
 dy_t = \sigma(y_t,y_{t-r_1}, \ldots, y_{t-r_k}) \, d x_t,  &  \qquad  t \in
 [0,T], \\ {} \,\,\,
y_t = \xi_t,  &\qquad t \in [-r,0], \end{array}
\right. \label{dde3}
\end{align}
 in the class of classical controlled
paths.  For this, it will be crucial
to use mappings of the type
$$ \Gamma :
\cq_{\kappa, \alpha}([a,b];\R^n)\times\cq_{\kappa, \tilde{\alpha}} ([a-r_k,
b-r_1]; \R^d) \rightarrow \cq_{\kappa, \alpha}([a,b];\R^n)$$
for $0 \leq a \leq b \leq T$, which are
defined by $(z,\tilde{z}) \mapsto \hz$, where $\hz_0=\al$ and $\der\hz$ given by
$\der\hz=\mathcal{J}(T_{\sigma}(z, \tilde{z}))$,
with $T_{\sigma}$ defined in Proposition \ref{compo:ccp-phi}.
>From now on, we will use the convention that $z_t=\tilde{z}_t=\hz_t=\xi_t$ for $t \in [-r,0]$.
Note that this convention is consistent with the definition of a
classical controlled path, see Definition \ref{def:ccp}: since
$\xi$ is $2\gamma$-H\"older continuous,  it can be considered as a part of
the remainder term $\rho$.

The first part of the current section will be devoted to the study of
the map $T$. By (\ref{bnd:norm-imdx2}) we have that
\begin{multline*}
\cn[ \mathcal{J}(T_{\sigma}(z, \tilde{z})) ;\cq_{\ka,\alpha}([a,b];\R^{n})]  \\
\le c_{norm}
\lp \| \sigma \|_{\infty} + c_{\textrm{\it int}}(b-a)^{\ga-\ka}\cn [ T_{\sigma}(z,
    \tilde{z});\cd_{\ka, \hat{\alpha}}([a,b];\R^{n,d}) ]\rp.
\end{multline*}
Since
$$
T_{\sigma}(z , \tilde{z}) =  \left( T_{\sigma^{(1)}}(z , \tilde{z}),\ldots, T_{\sigma^{(n)}}(z , \tilde{z})
\right)^*,
$$
where $\sigma^{(i)} \in C_b^3(\R^{n,d}; \R^{1,d})$ for $i=1, \ldots, n$
and $\sigma =(\sigma^{(1)},\ldots,\sigma^{(n)})^*$,
it follows  by  (\ref{bnd:t-phi-z})  that
\begin{eqnarray*}
&& \cn[T_{\sigma}(z, \tilde{z});\cd_{\ka,\hat{\alpha}}([a;b]; \R^{n})]
\\ && \qquad  \qquad  \le
 c_{\sigma,  T} \lp 1+ \cn^2[z;\cq_{\ka,\alpha}([a,b];\R^{n})]  +\cn^2[\tilde {z};\cq_{\ka,\tilde{\alpha}}([a-r_k,b-r_1];\R^{n})] \rp.
\end{eqnarray*}
Combining these two estimates we obtain
\begin{align}\label{est_cruc}
 &  \cn \left[ \Gamma(z, \tilde{z});\cq_{\ka,\alpha}([a;b]; \R^{n})
   \right ]  \\ &\quad  \leq c_{ \textrm{\it growth}} \left (   1+  \cn^2[\tilde
 {z};\cq_{\ka,\tilde{\alpha}}([a-r_k,b-r_1];\R^{n})]   \right)
\left (1 +  (b-a)^{\gamma- \kappa} \nonumber
   \cn^2[z;\cq_{\ka,\alpha}([a,b];\R^{n})] \right),
\end{align}
where the constant $c_{ \textrm{\it growth}}$ depends only on
$c_{\textrm{\it int}}$, $c_{norm}$, $\sigma$, $\kappa$, $\gamma$ and $T$.
Thus  the semi-norm  of the mapping $\Gamma$ is quadratically bounded
in terms of the semi-norm of $z$ and $\tilde{z}$.

\medskip

Now let $z^{(1)}, z^{(2)} \in \cq_{\kappa, \alpha}([a,b];\R^n)$ and
$\tilde{z} \in \cq_{\kappa, \tilde{ \alpha} } ([a-r_k,
b-r_1]; \R^d)$.  Then, by  (\ref{contract}) we have
\begin{align} \label{contract_1}
& \cn[\Gamma(z^{(1)}, \tilde{z}) - \Gamma(z^{(2)},
\tilde{z});\cq_{\ka,0}([a,b];\R^{n})] \\ & \qquad  \qquad \qquad
\qquad   \le   2 c_{norm} \, c_{\textrm{\it int}}  (b-a)^{\ga-\ka}\cn
  [ T_{\sigma}(z^{(1)}, \tilde{z}) - T_{\sigma}(z^{(2)},
  \tilde{z});\cd_{\ka,0}([a,b];\R^{n,d}) ] . \nonumber
\end{align}
 Applying  Proposition \ref{compo:ccp-loc-lin}, i.e. inequality
(\ref{contract_2}), to the right hand side of the above equation we obtain   that
\begin{align} \label{contract_f}
& \cn[\Gamma(z^{(1)}, \tilde{z}) - \Gamma(z^{(2)},
\tilde{z});\cq_{\ka,0}([a,b];\R^{n})]   \\ &
\qquad   \le   c_{ \textrm{\it lip}}  \big(1+C(z^{(1)},z^{(2)},\tilde{z})\big)^2 \,\cn
  [ z^{(1)} -z^{(2)};\cq_{\ka,0}([a,b];\R^{n,d}) ]
\,  (b-a)^{\ga-\ka}
\nonumber,  \end{align}
with a constant   $c_{ \textrm{\it lip}}$ depending only on
$c_{\textrm{\it int}}$, $c_{norm}$, $\sigma$, $\kappa$, $\gamma$ and $T$,
and moreover
\begin{align*}
C(z^{(1)},z^{(2)}, \tilde{z} )&=  \cn[ \tilde{z}
;\cq_{\ka,\tilde{\alpha}}([a-r_k,b-r_1];\R^{n})]
\\ & \qquad + \cn[z^{(1)} ;\cq_{\ka,\alpha}([a,b];\R^{n})] + \cn[z^{(2)}
;\cq_{\ka,\alpha}([a,b];\R^{n})].  \end{align*}
 Thus, for fixed $\tilde{z}$ the mappings $\Gamma(\cdot,\tilde{z})$ are locally
Lipschitz continuous with respect to the semi-norm  $\cn[ \cdot ;\cq_{\ka,0}([a,b];\R^{n})]$.

\bigskip

We also need the following Lemma, which can be shown by
straightforward calculations:

\medskip

\begin{lemma}{\label{lem_ineq}}
Let $c, \alpha  \geq 0$, $\tau \in [0,T]$ and define the set
$${\mathcal A}_{\tau}^{c, \alpha}=\{u\in\R_+^*:\,\,c( 1+\tau^{\alpha}u^2)\le u\}.$$
Set also
$\tau^{*} =   (8c^{2})^{-1/\alpha}.$
Then we have
${\mathcal A}_{\tau^*}^{c, \alpha} \neq \emptyset $
and
$ \sup \{u \, ; \,  u \in
{\mathcal A}_{\tau^*}^{c, \alpha} \} \leq (4+2\sqrt{2})c. $
\end{lemma}

\bigskip

Now we can state and prove our main result:
\begin{theorem}\label{thm:ex-uniq}
Let $x$ be a path satisfying Hypothesis \ref{hyp:x}, let $\xi \in
\cac^{2 \kappa}_1([-r,0]; \R^n)$  and let
$\si \in C^3_b(\R^{n,k+1}; \R^{n,d})$. Then we have:
\begin{enumerate}
\item
Equation (\ref{dde3}) admits a unique solution $y$ in
$\cq_{\ka,\xi_0}([0,T];\R^n)$ for any $\frac13<\ka<\ga$ and any $T>0$.
\item
 Let
$F: \cac_1^{2 \gamma}(\ro;\R^n)\times \cac_1^{\ga}([0,T];\R^d)\times
\left(\cac_2^{2\ga}([0,T];\R^{d\times d})\right)^{k+1} \rightarrow
  \cac_1^{\kappa}([0,T];\R^n)$  be the mapping defined by $$ F\left( \xi,x,\xd(0), \xd(-r_1),
\ldots , \xd(-r_k) \right) =y, $$
where $y$ is the unique solution of equation (\ref{dde3}).  This mapping
is locally
 Lipschitz continuous in the following sense:
Let $\tilde{x} $ be another driving rough path with corresponding
delayed L\'evy area $\xdt(-v)$, $v \in \{-r_k,\ldots,-r_0\},$ and $\tilde{\xi}$ another
initial condition. Moreover denote by
$ \tilde{y}$ the unique solution of the corresponding delay equation.
        Then, for every $N>0$, there exists a constant $K_N >0$ such that
\begin{align*} & \qquad \,\,\,  \| y - \tilde{y}\|_{\kappa,\infty}  \\ & \qquad
  \leq K_N \,    \left( \| x - \tilde{x}\|_{\gamma,\infty} +   \sum_{i=0}^{k}
\cn[  \xd(-r_i) - \xdt(-r_i) ; \cac_2^{2 \gamma}([0,T]; \R^{d})] +
\| \xi - \tilde{\xi} \|_{2 \gamma,\infty} \right) \end{align*}
holds for all tuples $(\xi,x, \xd, \xd(-r_1), \ldots, \xd(-r_k)), (\tilde{\xi},
\tilde{x}, \xdt, \xdt(-r_1), \ldots, \xdt(-r_k)) $ with
\begin{align*}
   \sum_{i=0}^{k}  \cn[  \xd(-r_i) ; \cac_2^{2 \gamma}([0,T]; \R^{d})]
&+ \sum_{i=0}^{k} \cn[  \xdt(-r_i) ; \cac_2^{2 \gamma}([0,T]; \R^{d})]
 \\  & + \| x\|_{\gamma,\infty} + \|\tilde{x}\|_{\gamma,\infty} +
 \| \xi \|_{2 \gamma,\infty} + \| \tilde{\xi} \|_{2 \gamma,\infty} \leq N , \end{align*}
where $\|f\|_{\mu,\infty}=\|f\|_{\infty}+|\der f|_{\mu}$ denotes  the usual H\"older norm of a path
$f$.
\end{enumerate}
\end{theorem}

\vspace{0.3cm}

\begin{proof}
The proof of Theorem \ref{thm:ex-uniq} is obtained by means of a fixed point argument, based
on the map $\Gamma$ defined above.

\medskip
\noindent
{\it 1) Existence and uniqueness.}
Without loss of generality assume that $T=N r_1$. We will
construct the solution of equation (\ref{dde3}) by induction over the
intervals
$[0, r_1]$, $[0, 2r_1]$, $\ldots$, $[0, N r_1],$
where we recall that $r_1$ is the smallest  delay in (\ref{dde3}).

\medskip
\noindent
{\it (i)} We will first show that equation (\ref{dde3}) has a solution
on the interval $[0,r_1]$. For this define
$$ \tilde{\tau}_1 =      (8c_1^{2})^{-1/(\ga-\ka)}\wedge r_1 ,
$$
where $$ c_1 = c_{ \textrm{\it growth}} \left (   1+  \cn^2[
 \xi ;\cac_2^{2 \gamma}([-r_k,0];\R^{n})]   \right).$$
Moreover, choose $\tau_1\in [0, \tilde{\tau}_1]$ and $N_1 \in
\mathbb{N}$ such that  $N_1 \tau_1 = r_1$, and define
$$ I_{i,1}=[(i-1) \tau_1, i \tau_1], \qquad i=1, \ldots, N_1.$$
Finally, consider the following mapping:
Let  $\Gamma_{1,1}:\cq_{\ka,\xi_{0}}(I_{1,1}; \R^n)\to\cq_{\ka,\xi_{0}}(I_{1,1}; \R^n)$ given
by $\hat{z}= \Gamma_{1,1}(z)$, where
$$    \,\, (\der\hz )_{st}=\cj_{st}(T_{\sigma}(z, \xi)\, dx)$$ for
  $\,\,  0 \leq s \leq t \leq  \tau_1$ .

Clearly, if $z^{(1,1)}$ is  a fixed point of the map $\Gamma_{(1,1)}$, then
$z^{(1,1)}$ solves equation (\ref{dde3}) on the interval $I_{1,1}$. We shall thus
prove that such a fixed point exists. First, due to (\ref{est_cruc})
we have the estimate
\begin{align}
\cn \left[ \Gamma_{1,1}(z);\cq_{\ka,\xi_{0}}(I_{1,1}; \R^{n}) \right]& \le
c_1 \lp 1 +  {\tau_1}^{\ga-\ka}   \cn^2 [z;\cq_{\ka,\xi_{0}}(I_{1,1}; \R^{n})] \rp.
\end{align}
Thanks to our choice of $\tau_1$ and Lemma \ref{lem_ineq}  we can now choose
$M_1 \in {\mathcal A}_{\tau^*}^{c_1,\ga-\ka}$ accordingly  and obtain that
the  ball \begin{equation}\label{def:ball-m}
B_{M_1}=\lcl z \in \cq_{\ka,\xi_{0}}(I_{1,1};\R^n);
\,\cn[z;\cq_{\ka,\xi_{0}}(I_{1,1};\R^n)]\le M_1 \rcl
\end{equation}  is left
invariant  under $\Gamma_{1,1}$.
Now, by changing $\tau_1$ to a smaller value
(and then $N_1$ accordingly) if necessary, observe that
$\Gamma_{1,1}$ also is a  contraction on  $B_{M_1}$, see (\ref{contract_f}).
Thus, the Banach theorem implies that the mapping
 $\Gamma_{1,1}$ has a fixed point, which
leads to a unique solution $z^{(1,1)}$ of equation (\ref{dde3}) on
the interval $I_{1,1}$.

If $\tau_1= r_1$, the first step of the proof is finished. Otherwise,
 define  the mapping  $\Gamma_{2,1}:\cq_{\ka,z^{(1,1)}_{\tau_1}}(I_{2,1}; \R^n)
\to\cq_{\ka,z^{(1,1)}_{\tau_1}}(I_{2,1}; \R^n)$
by $\hat{z}= \Gamma_{2,1}(z)$ where
$$    \,\, (\der\hz )_{st}=\cj_{st}(T_{\sigma}(z, \xi)\, dx)$$ for
  $\,\,  \tau_1 \leq s \leq t \leq  2\tau_1$  .
Since $\tau_1 < r_1$, it still holds
\begin{align} \label{est_cruc_2}
\cn [ \Gamma_{2,1}(z);\cq_{\ka,z^{(1,1)}_{\tau_1}}(I_{2,1}; \R^{n}) ]& \le
c_1 (1 +  {\tau_1}^{\ga-\ka}   \cn^2 [z;\cq_{\ka,z^{(1,1)}_{\tau_1}}(I_{2,1}; \R^{n})] )
\end{align}
and
we obtain by the same fixed point argument as above, the existence of a unique
solution  $z^{(2,1)}$
of equation (\ref{dde3}) on
the interval $I_{2,1}$.

Repeating this step as often as necessary, which is possible since the estimates
on the norms of the mappings $\Gamma_{j,1}$, $j=1, \ldots, N_1 $ are of the same
type as  (\ref{est_cruc}), i.e. the constant $c_1$ does not change, we obtain
that
$ z = \sum_{j=1}^{ N_1} z^{(j,1)} \,  \1_{I_{j,1}} $
is the unique solution to the equation  (\ref{dde3}) on the interval $[0, r_1]$.

\medskip

Now, it remains to verify that $z$ given as above is in fact a {\small CCP}.
First note that by construction $z$ is continuous on $[0,r_1]$ and
moreover that  $z$ is a {\small CCP}  on the
subintervals $I_{j,1}$ with decomposition $$ (\delta z)_{st}  =
\zeta^{(j,1)}_s (\delta x)_{st} + \rho^{(j,1)}_{st}, \qquad s,t \in I_{j,1},
$$ for $s \leq t$.
Clearly, we have
$$  (\delta z)_{st}  =   \sum_{j=j_s}^{j_t}  (\delta z^{(j,1)})_{s \vee t_j, t \wedge
t_{j+1}} , \qquad s,t \in
[0,r_1],$$ for $s \leq t$, where $t_j=(j-1)\tau_1$ and $j_s,j_t \in \{1,
\ldots, N-1 \}$ are such that
$$  t_{j_s} \leq  s < t_{j_s+1} < \ldots  < t_{j_t} <t \leq t_{j_t+1}.$$

Setting
$$ \zeta_s =  \sum_{j=1}^{ N_1} \zeta^{(j,1)}_s \,  \1_{I_{j,1}}(s) ,
\qquad s \in [0,r_1]$$
and
$$  \rho_{st} =  \sum_{j=j_s}^{j_t} (\zeta^{(j,1)}_{s \vee t_j} - \zeta^{(j_s,1)}_s) (\delta x)_{s \vee t_j, t \wedge
t_{j+1}} +  \sum_{j=j_s}^{j_t} \rho^{(j,1)}_{s \vee t_j, t \wedge
t_{j+1}}$$
we obtain
\begin{align*}  (\delta z)_{st} =  \zeta_s  (\delta x)_{st} + \rho_{st},
  \qquad s,t \in [0,r_1]
\end{align*}
for $ s \leq t$.

Now, it follows easily by the subadditivity of the H\"older norms that
$$
\sup_{s,t \in [0,\tau_1]}
\frac{|(\delta z)_{st}|}{{}\, \, |s-t|^{\kappa}}  \leq \sum_{j=1}^{N_1  }  \sup_{s,t \in I_{j,1}}
\frac{|(\delta z^{(j,1)})_{st}|}{{}\, \, |s-t|^{\kappa}}$$
and
$$ \sup_{t \in [0,\tau_1]} |\zeta_t |= \sup_{j=1,\ldots,N-1} \sup_{t \in I_{j,1}}
|\zeta_{t}^{(j,1)}|,
\qquad
\sup_{s,t \in [0,\tau_1]}
\frac{|(\delta \zeta)_{st}|}{{}\, \, |s-t|^{\kappa}}  \leq \sum_{j=1}^{N_1  }  \sup_{s,t \in I_{j,1}}
\frac{|(\delta \zeta^{(j,1)})_{st}|}{{}\, \, |s-t|^{\kappa}}.
$$
Furthermore, we obtain
$$
\sup_{s,t \in [0,\tau_1]}
\frac{|\rho_{st}|}{{}\, \, |s-t|^{2\kappa}}  \leq \sum_{j=1}^{N_1  }  \sup_{s,t \in I_{j,1}}
\frac{|(\rho^{(j,1)} )_{st}|}{{}\, \, |s-t|^{2 \kappa}} +   \sup_{s,t \in [0,\tau_1]}
\frac{| (\delta x)_{st} |}{{}\, \, |s-t|^{\kappa}} \sum_{j=1}^{N_1  }  \sup_{s,t \in I_{j,1}}
\frac{|(\delta \zeta^{(j,1)} )_{st}|}{{}\, \, |s-t|^{\kappa}}.
$$
Thus, we have in fact that $z \in \cq_{\kappa, \xi_{0}}([0, \tau_1]; \R^{n})$.

\medskip
\noindent
{\it (ii)} Let $l =1, \ldots, N-1$ assume that $\tilde{z}
\in  \cq_{\kappa, \xi_{0}}([0,  l r_1]; \R^{n})$ is the solution of the delay equation
(\ref{dde3}) on the interval
$[0, l r_1]$. Now we will construct the solution on the interval $[lr_1, (l+1)r_1]$.
Set $$c_{l+1}= c_{\textrm{\it growth}} \,  ( 1 +
  \cn^2[\tilde{z};\cq_{\ka, \tilde{z}_{l r_1 -r_k}} ( [lr_1 -r_k, lr_1]; \R^{n}]              )      $$
and define
$$ \tilde{\tau}_{l+1} =     (8c_{l+1}^{2})^{-1/(\ga-\ka)} \wedge r_1 .  $$
Furthermore, choose $\tau_{l+1}\in [0, \tilde{\tau}_{l+1}]$ and $N_{l+1} \in
\mathbb{N}$ such that  $N_{l+1} \tau_{l+1} = r_1$, and define
$$ I_{i,l+1}=  [  lr_1 + (i-1) \tau_{l+1} , l r_1 + i \tau_{l+1}], \qquad i=1, \ldots, N_{l+1}.$$
Consider   the mapping  $\Gamma_{1,l+1}:\cq_{\ka,\tilde{z}_{lr_1}}( I_{1,l+1}; \R^n)\to\cq_{\ka,\tilde{z}_{lr_1}}(I_{1,l+1}; \R^n) $
by $\hat{z}= \Gamma_{1,l+1    }(z) $ where
$$    \,\, (\der\hz )_{st}=\cj_{st}(T_{\sigma}(z, \tilde{z})\, dx)$$ for
  $\,\,  lr_1 \leq s \leq t \leq  lr_1+\tau_{l+1}$  .
Again  $z^{(1,l+1)}$ is  a fixed point of the map $\Gamma_{1,l+1}$ if and only
if $z^{(1,l+1)}$ solves equation (\ref{dde3}) on the interval $I_{1,l+1}$.
However, by (\ref{est_cruc})
we have the estimate
\begin{align*}
\cn [ \Gamma_{1,l+1} (z);\cq_{\ka,   \tilde{z}_{l r_1} }( I_{1,l+1} ; \R^{n}) ] &
   \leq c_{l+1} \, ( 1+   {\tau_{l+1}}^{\ga-\ka}   \cn^2  [z;\cq_{\ka,   \tilde{z}_{lr_1} }(I_{1,l+1} ; \R^{n}) ] ).
\end{align*}
Now we can apply the same fixed point argument as in step {\it (i)},
 which leads to a unique solution $z^{(1,l+1)}$ of (\ref{dde3}) on the interval $I_{1,l+1} $.

If $\tau_{l+1} \neq r_1$, define for the next interval $I_{2,l+1}$ the mapping
 $$\Gamma_{2,l+1}:\cq_{\ka,z^{(1,l+1)}_{lr_1+ \tau_l}}(I_{2,l+1} ; \R^n)\to\cq_{\ka,z^{(1,l+1)}_{lr_1+\tau_l}}( I_{2,l+1}; \R^n)$$
by $\hat{z}= \Gamma_{ 2,l+1} (z)$, where
$(\der\hz)_{st}=\cj_{st}(T_{\sigma}(z, \tilde{z})\, dx)$ for
$lr_1  + \tau_{l+1}\leq s \leq t \leq  lr_1+2 \tau_{l+1}$  .
Since $ lr_1 +\tau_{l+1} \leq (l+1)r_1$,
we still have the estimate
\begin{align*}
& \cn [ \Gamma_{2,l+1}(z);\cq_{\ka,   z^{(1,l+1)}_{lr_1 + \tau_l } }(
  I_{2,l+1}; \R^{n}) ]  \leq
c_{l+1} \, ( 1+   \tau_{l+1}^{\ga-\ka}   \cn^2  [z;\cq_{\ka,   z^{(1,l+1)}_{lr_1 + \tau_l }}(I_{2,l+1}; \R^{n}) ] ).
\end{align*}
Now the existence of a unique solution $z^{(2,l+1)}$ of (\ref{dde3}) on the interval $I_{2,l+1}$ follows again by the same fixed point argument.

Proceeding completely analogous to step {\it (i)} we obtain the existence of a
unique path $z \in \cq_{\kappa, \tilde{z}_{lr_1}}([lr_1, (l+1)r_1]; \R^{n})$,
which solves the delay equation (\ref{dde3}) on the interval $[lr_1, (l+1)r_1]$
for a given ``initial path''
   $\tilde{z} \in  \cq_{\kappa, \xi_{0}}([0,  l r_1]; \R^{n})$.
Patching these two paths together, we obtain (using the same arguments as at the end of  step
 {\it(i)}) a path  $z \in  \cq_{\kappa, \xi_{0}}([0,  (l+1) r_1]; \R^{n})$,
which solves equation (\ref{dde3}) on the interval $[0,  (l+1) r_1]$.

Thus we have shown that there exists a unique path $z \in  \cq_{\kappa,
  \xi_{0}}([0,  T]; \R^{n})$, which is a solution of the equation
(\ref{dde3}). Moreover, by the above construction we obtain the
following bound on the norm of this path:
\begin{align} \label{bound_sol}
&\cn[ z; \cq_{\kappa, \xi_{0}}([0, T]; \R^{n})] \\ & \quad \leq f \lp \cn[ x;
  \cac^{\gamma}_1([0,T]; \R^{n})] + \sum_{i=0}^{k} \cn[ \xd(-r_i);
  \cac^{2\gamma}_2([0,T]; \R^{n})] + \cn[ \xi;
  \cac^{2\gamma}_1([0,T]; \R^{n})] \nonumber
      \rp,
\end{align} where $f:[0, \infty) \rightarrow (0, \infty)$ is a
continuous non-decreasing function, which depends only on $\kappa, \gamma, n,d, \sigma,
T$ and $r_1, \ldots, r_k$.

\medskip
\noindent
{\it 2) Continuity of the It\^{o} map.}
Let $y= F\left( \xi,x,\xd(0), \xd(-r_1),
\ldots , \xd(-r_k) \right).$
Since $y$ solves equation (\ref{dde3}), we have
$ (\delta y)_{st} = \cj_{st}(\sigma(y_s, \mathfrak{s}(y)) \, dx_s).$
It follows by  the Propositions \ref{compo:ccp-phi} and  \ref{intg:mdx} that
\begin{equation}
(\der y)_{st}= \label{decomp_y_1}
m_s ( \der x)_{st}
+ \sum_{i=0}^{k}
\zeta^{(i)}_s \cdot \xdst(-r_i)
+\laa_{st}\lp  \rho \der x +  \sum_{i=0}^{k}
\der\zeta^{(i)} \cdot \xd(-r_i) \rp
\end{equation}
for $0 \leq s \leq t \leq T$, with
\begin{equation}
m_s = \sigma(y_s, \mathfrak{s}(y)_s), \quad  \zeta^{(i)}_s=
\psi_s^{(i)} m_{s-r_i}, \quad
  \psi_s^{(i)} =    \left(  \frac{ \partial \vp }{
      \partial x_{1,i}}   (y_s, \mathfrak{s}(y)_{s})
 , \ldots  ,  \frac{ \partial \vp }{ \partial x_{n,i}}
  (y_s, \mathfrak{s}(y)_{s}) \right)
\label{61}
\end{equation}
 for $i=0, \ldots, k$.
Moreover, note that the remainder term $\rho$ of the decomposition of $y$ satisfies the relation
\begin{equation}
 \rho_{st} =  \sum_{i=0}^{k}
\zeta^{(i)}_s \cdot \xdst(-r_i)
+\laa_{st}\lp \rho \der x +  \sum_{i=0}^{k}
\der\zeta^{(i)} \cdot \xd(-r_i) \rp.
\label{62}
\end{equation}
\medskip

Now consider (\ref{dde3}) with a different initial path $\tilde{\xi}$,
driving rough path $\tilde{x}$ and corresponding delayed L\'evy area $\xdt(v)$, for $v \in
   \{-r_k,\ldots,-r_0\}$. If the assumptions of the theorem are  satisfied, then also
the equation
\begin{equation*}
\left\{
\begin{array}{lll}
 d \tilde{y}_t &= \sigma(\tilde{y}_t, \tilde{y}_{t-r_1}, \ldots,
\tilde{y}_{t-r_k}) \, d \tilde{x}_t , \qquad t \in [0,T],\\
\tilde{y}_t &= \tilde{\xi}_t,  \qquad \qquad \qquad \qquad \qquad \,\, \quad t \in [-r,0]
\end{array}\right.
\end{equation*}
admits a unique solution  $\tilde{y}= F( \tilde{\xi}, \tilde{x}(0),\xdt, \xdt(-r_1),
\ldots , \xdt(-r_k) )$.
Clearly we also have in this case
\begin{equation} \label{decomp_y_2}
(\der \tilde{y})_{st}=
\tilde{m}_s ( \der \tilde{x})_{st}
+ \sum_{i=0}^{k}
\tilde{\zeta}^{(i)}_s \cdot (\xdt(-r_i))_{st}
+\laa_{st}\lp \tilde{\rho} \der \tilde{x} +  \sum_{i=0}^{k}
\der\tilde{\zeta}^{(i)} \cdot \xdt(-r_i) \rp
\end{equation}
for $0 \leq s \leq t \leq T$, with
$\tilde{m}$, $\tilde{\zeta}^{(i)}$ and $  \tilde{\psi}^{(i)}$
defined according to (\ref{61}) and (\ref{62}).

\medskip
\noindent
{\it (i)}  We first analyse the difference between $\rho$ and $\tilde{\rho}$.
Here we have
\begin{equation}\label{eq:def-rho-minus-tilde-rho}
\rho_{st} -\tilde{\rho}_{st} = e^{(1)}_{st} + \Lambda_{st}(e^{(2)})      ,
\end{equation}
with
\begin{eqnarray*}
e^{(1)}_{st}&=& \sum_{i=0}^{k}
\zeta^{(i)}_s \cdot (\xd(-r_i))_{st}
- \sum_{i=0}^{k}
\tilde{\zeta}^{(i)}_s \cdot (\xdt(-r_i))_{st}  \\
e^{(2)} &=& \rho\,\delta x - \tilde{\rho}\,\delta \tilde{x}  +\sum_{i=0}^{k}
\der\zeta^{(i)} \cdot \xd(-r_i) -
\sum_{i=0}^{k}
\der\tilde{\zeta}^{(i)} \cdot \xdt(-r_i) .
\end{eqnarray*}
Now
set
\begin{multline*}
C(y)=   \|x\|_\infty + \|x\|_\ga
 + \sum_{i=0}^{k} \| \xd(-r_i)\|_{2\ga}  +
\cn[y; \cq_{\kappa, \alpha}([0,T]; \R^{n})] +  \| \xi\|_\infty + \|\xi\|_{2\gamma},
\end{multline*}
define $C(\tilde{y})$ accordingly for $\tilde y$, and let $R$ be the quantity
\begin{multline*}
R=\| x- \tilde{x}\|_\infty +
\|x-\tilde{x}\|_{\ga} +\sum_{i=0}^{k} \| \xd(-r_i) -
  \xdt(-r_i)\|_{2\gamma}  + \| \xi- \tilde{\xi}\|_\infty +
 \|\xi-\tilde{\xi}\|_{2\gamma}.
\end{multline*}
In the following we will denote
constants, which depend only on $\kappa, \gamma, n, d,\sigma$ and $T$, by
$c$ regardless of their value.

\smallskip

Fix an interval $[a,b] \subset [0,T]$.
By straightforward calculations we have
\begin{align} \label{est_e1}
|e^{(1)}_{st}| &\leq c \,  (1+C(y))  |t-s|^{2\gamma}  \, R+ c \, C(y)
|t-s|^{2\gamma}  \,  \sup_{\tau \in [(s-r_k)^+, t]} |y_{\tau}- \tilde{y}_{\tau}|
\end{align}
for $s,t \in [a,b]$.
Now, consider the term $e^{(2)}$. We have
\begin{align*} e_{sut}^{(2)} &= \rho_{su} (\der x)_{ut} - \tilde{\rho}_{su} (\der \tilde{x})_{ut} +  \sum_{i=0}^{k}
(\der \zeta^{(i)})_{su} \cdot {\bf x}^{\bf 2}_{ut}(-r_i) - \sum_{i=0}^{k}
(\der \tilde{\zeta}^{(i)})_{su} \cdot \xdt_{ut}(-r_i)
\\ & = (\rho- \tilde{\rho})_{su}  (\der x)_{ut} + \tilde{\rho}_{su}  (\der
(x - \tilde{x} ))_{ut}
\\& \qquad + \sum_{i=0}^{k}
(\der (\zeta^{(i)} -\tilde{\zeta}^{(i)}))_{su} \cdot {\bf x}^{\bf 2}_{ut}(-r_i) - \sum_{i=0}^{k}
(\der \tilde{\zeta}^{(i)})_{su} \cdot (  {\bf x}^{\bf 2}_{ut}(-r_i)    - \xdt_{ut}(-r_i))
\end{align*}
for $s,u,t \in [a,b]$.
Clearly, it holds
\begin{align*}
\left|  (\rho- \tilde{\rho})_{su}  (\der x)_{ut}  \right| & \leq C(y)
|t-u|^{\gamma}  |s-u|^{2
  \kappa}
\, \cn[  \rho- \tilde{\rho}; \cac_2^{2
  \kappa}([a,b]; \R^{n} )     ]  , \\
\left| \tilde{\rho}_{su}  (\der
(x - \tilde{x} ))_{ut} \right| &\leq |t-u|^{\gamma}  |s-u|^{2
  \kappa} \,  \cn[  \tilde{\rho}; \cac_2^{2
  \kappa}([a,b]; \R^{n} )     ] \,  R
 \end{align*}
and
\begin{align*}
 \left| \sum_{i=0}^{k}
(\der (\zeta^{(i)} -\tilde{\zeta}^{(i)}))_{su} \cdot {\bf x}^{\bf 2}_{ut}(-r_i)
\right| & \leq c \, C(y)  |t-u|^{2\gamma} \,  \sum_{i=0}^{k} \left|
(\der (\zeta^{(i)} -\tilde{\zeta}^{(i)}))_{su} \right|, \\ \left| \sum_{i=0}^{k}
(\der \tilde{\zeta}^{(i)})_{su} \cdot (  {\bf x}^{\bf 2}_{ut}(-r_i)    -
\xdt_{ut}(-r_i))  \right| & \leq  |t-u|^{2 \gamma} \, R \,
\sum_{i=0}^{k} \left|
(\der \tilde{\zeta}^{(i)})_{su} \right| .
\end{align*}
Furthermore, we also have, for any $i=0,\ldots,k$ that
$$
\left|
\der (\zeta^{(i)}- \tilde{\zeta}^{(i)}  )_{su} \right|  \leq
c \sup_{\tau_1,\tau_2 \in [s-r_i,t-r_i]} | (y_{\tau_1}-\tilde{y}_{\tau_1})
-(y_{\tau_2} - \tilde{y}_{\tau_2} )|,
\quad \,\,
\left|
(\der \tilde{\zeta}^{(i)} )_{su} \right|  \leq
C(\tilde{y})  |s-u|^{\kappa}.
$$

Recall that the H\"older norm of a path $f$ is defined by
$$ \|  f\|_{\mu, \infty,[s,t]} =  \sup_{ \tau \in [s,t]}
|f_{\tau}| +  \sup_{ \tau_1, \tau_2 \in [s,t]} \frac{
|f_{\tau_1} -f_{\tau_2} |
}{|\tau_2 - \tau_1|^{\mu}}.
$$
Set also $C=c(1+C(y)+ C(\tilde{y}))$,
where $c$ is again an arbitrary constant depending only on $\kappa,
\gamma, n, d,
\sigma$ and $T$.
Using these notations and combining the previous estimates, we end up with:
\begin{multline}\label{est_e2}
\left| e^{(2)}_{sut} \right|  \leq  C |t-u|^{\gamma} |s-u|^{2 \kappa}    \,
R+
 C  |t-u|^{2\gamma}
|s-u|^{ \kappa} \sum_{i=0}^k \|y- \tilde{y} \|_{\kappa,\infty,  [a-r_i,b-r_i]}   \\
+ C  |t-u|^{\gamma}  |s-u|^{ 2
  \kappa} \, \cn[  \rho- \tilde{\rho}; \cac_2^{2
  \kappa}([a,b]; \R^{n} )     ]   .
\end{multline}
Hence $e^{(2)}$ belongs to $ \dom(\laa) $
and we obtain by Proposition \ref{prop:Lambda}
that
\begin{equation}
 \|   \Lambda( e^{(2)}) \|_{ 3 \kappa } \label{est_e3}
 \leq  C \, R  + C \,\sum_{i=0}^k
\|y- \tilde{y} \|_{\kappa,\infty,  [a-r_i,b-r_i]}  + C  \,  \cn[  \rho- \tilde{\rho}; \cac_2^{2
  \kappa}([a,b]; \R^{n} )     ].
\end{equation}
Inserting the estimates for $e^{(1)}$ and $\Lambda(e^{(2)})$,
i.e. (\ref{est_e1}) and (\ref{est_e3}), into the definition (\ref{eq:def-rho-minus-tilde-rho})
of $\rho - \tilde{\rho}$ gives finally
\begin{eqnarray*}
&&\cn [ \rho - \tilde{\rho}; \cac_{2}^{2\kappa}( [a,b]; \R^{n})] \leq  C
|b-a|^{\gamma-\kappa} \sum_{i=0}^k  \| y - \tilde{y} \|_{\kappa, \infty,[a-r_i,b-r_i]}
\, R
\\
&&\hskip5cm+ C
 |b-a|^{\gamma-\kappa}   \, R
 +  C |b-a|^{\gamma-\kappa} \,  \cn[  \rho- \tilde{\rho}; \cac_2^{2
  \kappa}([a,b]; \R^{n} )     ]  , \nonumber
\end{eqnarray*}
and due to the subadditivity of the H\"older norms, we get
\begin{multline}\label{est_finale_1}
\cn [ \rho - \tilde{\rho}; \cac_{2}^{2\kappa}( [a,b]; \R^{n})] \leq  C |b-a|^{\gamma-\kappa} \,  \cn[  \rho- \tilde{\rho}; \cac_2^{2
  \kappa}([a,b]; \R^{n} )     ]  +
C
|b-a|^{\gamma-\kappa}   \,  \| y - \tilde{y} \|_{\kappa, \infty,[a,b]} \,
R\\
+ C
|b-a|^{\gamma-\kappa}   \, \| y -
\tilde{y} \|_{\kappa, \infty,[a-r_k,a]}  \, R   + C  |b-a|^{\gamma-\kappa}
 \,R.
\end{multline}

\medskip
\noindent
{\it (ii)}
Now consider the difference between $y$ and $\tilde{y}$.
Completely analogous to step {\it (i)} we also obtain that
\begin{multline*}
 \cn [ y- \tilde{y}; \cac_{1}^{\kappa}( [a,b]; \R^{n})] \leq
C |b-a|^{\gamma-\kappa}  \,  \cn[  \rho- \tilde{\rho}; \cac_2^{2
  \kappa}([a,b]; \R^{n} )     ]+
C
|b-a|^{\gamma-\kappa}   \,  \| y - \tilde{y} \|_{\kappa,\infty, [a,b]} \,
R\\
+ C
|b-a|^{\gamma-\kappa}   \, \| y -
\tilde{y} \|_{\kappa, \infty,[a-r_k,a]} \, R   + C  |b-a|^{\gamma-\kappa}
\,R.
\end{multline*}
Moreover, since
$$  \sup_{\tau \in [a,b]} |y_{\tau} - \tilde{y}_{\tau}| \leq   |y_{a}
-\tilde{y}_a| +   (b-a)^{\kappa} \,
\cn [ y - \tilde{y}; \cac_{1}^{\kappa}( [a,b]; \R^{n})], $$
we also have
\begin{multline} \label{est_finale_2}
 \| y - \tilde{y} \|_{\kappa; [a,b]}  \leq  C |b-a|^{\gamma-\kappa} \,
\cn[  \rho- \tilde{\rho}; \cac_2^{2
  \kappa}([a,b]; \R^{n} )     ]  +
C
|b-a|^{\gamma-\kappa}   \, \| y - \tilde{y} \|_{\kappa, \infty,[a,b]} \, R \\
+ C
|b-a|^{\gamma-\kappa}  \,  \| y -
\tilde{y} \|_{\kappa, \infty,[a-r_k,a]} \, R   + |y_a- \tilde{y}_a|+ C
|b-a|^{\gamma-\kappa}  \,R.
\end{multline}

\medskip
\noindent
{\it (iii)} Now set
$$ \Delta(a,b)= \cn [ \rho - \tilde{\rho}; \cac_{2}^{2\kappa}( [a,b]; \R^{n})]
+ \| y -
 \tilde{y} \|_{\kappa,\infty, [a,b]}. $$
By combining (\ref{est_finale_1}) and (\ref{est_finale_2})  we
finally have that
\begin{multline} \label{est_finale_3}
  \Delta(a,b)    \leq  C(1+R)  |b-a|^{\gamma-\kappa} \,  \Delta(a,b)
 + C(1+ R)
|b-a|^{\gamma-\kappa}  \, \Delta((a-r_k)^+,a)  \\     + |y_a- \tilde{y}_a| +
C(1+R)  |b-a|^{\gamma-\kappa} \, R.
\end{multline}
Now
choose $a=0$ and $b_1 = \left(\frac{1}{2 C(1+R)}
\right)^{1/(\gamma-\kappa)} $.
In this case, we obtain from (\ref{est_finale_3}) that
\begin{align*}
\Delta(0,b_1) \leq \frac{1}{2}  \Delta(0,b_1)  +
|\xi_0 - \tilde{\xi_0}| + \frac{1}{2} \,R ,
\end{align*}
which yields
\begin{align}
\Delta(0,b_1) \leq    R
+ 2  |\xi_0 - \tilde{\xi_0}|  \leq
3 R.  \label{est_almost_fin}
\end{align}
For the next interval $[b_1,2b_1]$, we obtain in turn that
\begin{align*} \Delta(b_1,2b_1)
 \leq \frac{1}{2} \Delta(b_1,2b_1)  + \frac{1}{2}  \Delta(0,b_1)  + |y_{b_1} - \tilde{y}_{b_1}|
 +  \frac{1}{2}R ,
\end{align*}
and hence
\begin{eqnarray*}
\Delta(b_1,2b_1)
 &\leq & \Delta(0,b_1) +  2 |y_{b_1} - \tilde{y}_{b_1}| + R \le 10R,
\end{eqnarray*}
by  (\ref{est_almost_fin}).

Repeating this step $ \lfloor T/ b_1 \rfloor$-times
we obtain that there exists a continuous non-de\-crea\-sing function $g: (0, \infty)
\rightarrow (0, \infty)$ such that
$$ \Delta(ib_1, (i+1)b_1) \leq g(T/ b_1   ) \, R$$
for all $i= 0, \ldots,   \lfloor  T/ b_1 \rfloor$.
 Using the subadditivity of
the H\"older norms,
we obtain  the estimate
\begin{align} \Delta(0,T) \leq  (1+T/b_1 ) g(T/ b_1   ) \,
  R. \label{el_1} \end{align}
Now recall that   $C=c(1+C(y) + C(\tilde{y}))$ and note that $ R \leq
c(C(y) + C(\tilde{y}))$. Thus we  have
 $$T/b_1 = T \left(2 C(1+R)
\right)^{1/(\gamma-\kappa)}   \leq c ( C(y) + C(\tilde{y})))^{1/(\gamma-\kappa)},$$
where
\begin{multline*}
C(y)=   \|x\|_\infty + \|x\|_\ga
 + \sum_{i=0}^{k} \| \xd(-r_i)\|_{2\ga}  +
\cn[y; \cq_{\kappa, \alpha}([0,T]; \R^{n})] +  \| \xi\|_\infty + \|\xi\|_{2\gamma},
\end{multline*}
and $C(\tilde{y})$ is defined accordingly. However, by
(\ref{bound_sol}) it  follows that
\begin{align*} C(y) +C(\tilde{y}) \leq D +f(D) + \tilde{D} + f(\tilde{D}) ,
\end{align*} where $$D=  \|x\|_\infty + \|x\|_\ga
 + \sum_{i=0}^{k} \| \xd(-r_i)\|_{2\ga}  +  \| \xi\|_\infty + \|\xi\|_{2\gamma},$$
and $\tilde{D}$ is again defined accordingly.
Thus, we obtain now from (\ref{el_1}) that there exists a continuous
function $ \bar{g}: [0, \infty) \rightarrow [0, \infty)$, which depends
only on   $\kappa, \gamma, \sigma, n, d, T$ and $r_1, \ldots, r_k$, such  that
$$  \Delta(0,T) \leq \bar{g}(D+ \tilde{D}) \, R  .$$
Hence, the assertion follows.

\end{proof}

\section{Application to the fractional Brownian motion}

All the previous constructions rely on the specific assumptions
we have made on the path $x$. In this
section, we will show how
our results can be applied to the fractional Brownian motion.

\subsection{Definition}
We consider in this section a $d$-dimensional fBm with
Hurst parameter $H$ defined on the real line, that is a
centered Gaussian process
\begin{equation*}
B=\lcl B_t=(B_t^1,\ldots,B_t^d);\, t\in\R \rcl,
\end{equation*}
where $B^1,\ldots,B^d$ are $d$ independent one-dimensional
fBm, i.e., each $B^{i}$ is a centered Gaussian process with continuous sample paths and  covariance function
\begin{equation}
\label{rhts}
R_H(t,s)=
\frac12 \lp |s|^{2H}+|t|^{2H}-|t-s|^{2H} \rp
\end{equation}
for $i=1, \ldots, d$.
The fBm verifies the following  two important properties:
\begin{equation}\label{scaling}
\mbox{(scaling)}\quad \mbox{For any }c>0,\,B^{(c)}=c^H B_{\cdot/c}\mbox{ is a fBm},
\end{equation}
\begin{equation}\label{stationarity}
\mbox{(stationarity)}\quad\mbox{For any }h\in\R,\,B_{\cdot+h}-B_h\mbox{ is a fBm}.
\end{equation}
Notice that, for Malliavin calculus purposes, we shall assume in the sequel that $B$ is
defined a complete probability space $(\Omega,\cf,P)$, and
that $\cf=\si(B_s;\, s\in\R)$.
Observe also that we work with a fBm indexed by $\R$
for sake of simplicity, since this  allows some more elegant calculations for the
definition of the delayed L\'evy area.

\subsection{Malliavin calculus with respect to fBm}

Let us give a few facts about the Gaussian structure of fractional Brownian motion
and its Malliavin derivative process, following Section 2 of \cite{nual-cours}.
Let $\EE$ be the set of step-functions on $\R$ with values in $\R^{d}$.
Consider the Hilbert space $\HH$ defined as the closure of $\EE$ with respect
to the scalar product induced by
\begin{eqnarray*}
&& \left\lpa
(\1_{[t_{1},t^{1}]}, \ldots, \1_{[t_{d},t^{d}]})
,
(\1_{[s_{1},s^{1}]}, \ldots, \1_{[s_{d},s^{d}]})
\right\rpa_{\HH} \\
&&\hskip2cm = \sum_{i=1}^{d}
\big(
R_H(t^{i},s^{i})-R_H(t^{i},s_i)-R_H(t_i,s^i)+R_H(t_i,s_i)
\big),
\end{eqnarray*}
for any $-\infty<s_{i}<s^i<+\infty$ and $-\infty<t_{i}<t^i<+\infty$,
and where $R_{H}(t,s)$ is given by (\ref{rhts}).
The mapping
$$ (\1_{[t_{1},t^{1}]}, \ldots, \1_{[t_{d},t^{d}]}) \mapsto \sum_{i=1}^{d}\big(
B_{t^{i}}^{i}  - B_{t_i}^i\big) $$can be
extended to an isometry between $\HH$ and the Gaussian space $H_{1}(B)$ associated with
$B=(B^{1}, \ldots , B^{d})$. We denote this isometry by $\varphi \mapsto B(\varphi)$.
Let $\mathcal{S}$ be the set of smooth cylindrical random variables of the form
$$F = f(B(\varphi_{1}), \ldots, B(\varphi_{k})), \qquad \varphi_{i} \in \HH, \quad
i=1, \ldots, k,$$ where $f\in C^{\infty}(\R^{d,k},\R)$ is bounded with bounded derivatives. The derivative operator $D$ of a smooth cylindrical random variable of the above form is defined as the $\HH$-valued random variable
$$ DF= \sum_{i=1}^{k} \frac{\partial f}{\partial x_{i}} (B(\varphi_{1}), \ldots, B(\varphi_{k})) \varphi_{i}.$$
This operator is closable from $L^{p}(\Omega)$ into $L^{p}(\Omega; \HH)$.
 As usual, $\sk^{1,2}$ denotes
the closure of the set of smooth random variables with respect to the norm
$$\| F\|_{1,2}^2 \; = \; { \rm E} |F|^2  + { \rm E} \| DF\|_{\HH}^2 .$$
In particular, if $D^{i}F$ denotes the Malliavin derivative
of  $F\in\sk^{1,2}$ with respect to $B^i$,
we have   $D^{i} B^{j}_{t} = \delta_{i,j}{\bf 1}_{[0,t]}$ for $i,j=1, \ldots, d$.

The divergence operator $I$ is the adjoint of the derivative operator. If a random variable $u \in L^{2}(\Omega; \HH)$
belongs to $\textrm{dom}(I)$, the domain of the divergence operator, then  $I(u)$ is defined by the
duality relationship
\begin{align} \label{sk_dual}
 { \rm E} (F \,I(u))= { \rm E} \langle D F, u \rangle_{\HH}, \end{align} for every $F \in \sk^{1,2}$.
Moreover, let us recall two useful properties verified by $D$ and $I$:
\begin{itemize}
\item If $u \in \textrm{dom}(I)$ and $F \in \sk^{1,2}$ such that $Fu \in L^{2}(\Omega; \HH)$, then we have the following integration by parts formula:
\begin{equation}\label{int_part}    I(Fu)=FI(u)- \langle DF, u \rangle_{\HH}. \end{equation}
\item If $u$ verifies $ {\rm E}  \| u \|_{ \HH}^2 +
{\rm E} \| Du \|_{ \HH \otimes \HH}^2 < \infty$, $D_r u \in  \textrm{dom}(I)$ for
all $r \in \R$ and
$  (I(D_r u))_{r \in \R}$ is an element of $L^{2}(\Omega; \HH)$,  then
\begin{equation}\label{mal-sko}
D_rI(u)= u_r + I(D_r u).
\end{equation}
\end{itemize}

\subsection{Delayed L\'evy area and fractional Brownian motion}

The stochastic integrals we shall use in order to define our delayed L\'evy area are defined,
in a natural way, by Russo-Vallois symmetric approximations, that is, for a given
process $\phi$:
$$
\int_s^t \phi_w \, d^\circ B^i_w \,=L^2-\lim_{\ep\rightarrow 0} \,(2\ep)^{-1}
\int_s^t \phi_w \, \big(B^i_{w+\ep}-B^i_{w-\ep}\big)dw,
$$
provided the limit exists. This pathwise type notion of integral can then
be related to some stochastic analysis criterions in the following way
(for a proof, see \cite{ALN}):
\begin{theorem}\label{thm_help}
Fix $t\ge 0$ and let $\phi\in\mathbb{D}^{1,2}({\mathcal H})$ be a process such that
$$
{\rm Tr}_{[0,t]}D^{B^i}\phi:=L^2-\lim_{\ep\rightarrow 0}(2\ep)^{-1}\int_0^{t}
\langle D^{B^i}\phi_u,{\bf 1}_{[u-\ep,u+\ep]}\rangle_{\mathcal H} du
$$
exists, and such that, setting $\ell(t,u)\triangleq u^{2H-1}+(t-u)^{2H-1}$
for $0\le u < t$,
$$
\int_0^{t} \!\!\!E[\phi_u^2]\,\ell(t,u) du
+\int_{[0,t]^2} \!\!\!\!\!\!E\lc (D^{B^i}_r\phi_u)^2\rc \,\ell(t,u) dudr
<\infty.
$$
Then $\int_0^t \phi d^\circ B^i$ exists, and verifies
$$
\int_0^t \phi d^\circ B^i=I^{B^i}(\phi\1_{[0,t]})+{\rm Tr}_{[0,t]}D^{B^i}\phi.
$$
\end{theorem}

With these notations in mind, the main result of this section is the following:
\begin{proposition}\label{prop:hyp-fbm}
Let $B$ be a $d$-dimensional fractional Brownian motion
and suppose $H > \frac{1}{3}$. For $v\in[-r,0]$,
let $\bd(v)$ be the delayed L\'evy area given by:
$$
\bdst(v)=\ist dB_u \otimes \int_{s+v}^{u+v} dB_r,
\ \mbox{i. e.}\
\bdst(v)(i,j)=\ist dB_u^i  \int_{s+v}^{u+v} dB_r^j,
\quad i,j\in\{1,\ldots,d  \},
$$
Then almost all sample paths of  $B$ satisfy
Hypothesis \ref{hyp:x}.
\end{proposition}

\noindent{\it Proof}.
When $H=\frac12$, the desired conclusion is easily obtained, because
the Russo-Vallois symmetric integral coincides with the Stratonovich
integral. Moreover, for $H>\frac{1}{2}$ the  Russo-Vallois symmetric
integral coincides with the Young integral, which is  well defined in
this case, and the assertion still follows easily from the properties of
Young integrals.

Now, fix $\frac13<H< \frac12$.
It is a classical fact
that $B\in\cac_1^\ga([0,T];\R^d)$ for any $\frac13<\ga<H$.
Due to the stationarity property (\ref{stationarity})
we will work without loss of generality on the interval $[0,t-s]$
instead of $[s,t]$ in the sequel.

\vspace{0.3cm}

\noindent
{\it 1) Case $i=j$}.
When $v=0$, it is easily checked that
$$
E|\bdst(0)(i,i)|^2=\frac14\,E|B_t-B_s|^4=\frac34|t-s|^{4H}.
$$
Let us now consider the case where $v<0$.
For $\phi=(B^{i}_{\cdot+v}-B^{i}_{v}) \1_{[0,t-s]}(\cdot)$,
the conditions of Theorem \ref{thm_help} are easily verified, hence
$\int_{0}^{t-s} \phi_u d^\circ B^{i}_u$ exists.
Notice moreover that we have
$D^{B^i}_{r}\phi_{u} = {\bf 1}_{[v,u+ v]}(r){\bf 1}_{[0,t-s]}(u)$
and, for $u\in[0,t-s]$ and $\e \in [0, -v]$ (which is always the case,
for a {\it fixed} $v<0$ and $\e$ small enough) it holds
\begin{align*}
&\langle   {\bf 1}_{[v,u+v]}
, {\bf 1}_{[ u-\ep, u+\ep]}
\rangle_{\HH} =  \frac{1}{2} \left( | v+ \ep |^{2H} -  | v - \ep
  |^{2H}  + | v - u- \ep |^{2H}  -  | v - u + \ep |^{2H}  \right)\\
&= \frac{1}{2} \left( (-v- \ep )^{2H} -  ( - v + \ep
  )^{2H}  + (- v + u+ \ep )^{2H}  -  (- v + u - \ep )^{2H}  \right).
\end{align*}
Thus, we obtain
$$
{\rm Tr}_{[0,t-s]}D^{B^i}\phi=-H(-v)^{2H-1} (t-s)+\frac{1}{2} \left( (t-s-v)^{2H}
   - (-v)^{2H} \right).
$$
For $x\ge 0$, it is well-known that $0\le ((-v)+x)^{2H}-(-v)^{2H}\le 2H(-v)^{2H-1}x$. Applying this inequality to the
second term of the right hand side of ${\rm Tr}_{[0,t-s]}D^{B^i}\phi$, we get
\begin{eqnarray}{\label{taylor_rem}}
\left| {\rm Tr}_{[0,t-s]}D^{B^i}\phi \right| \leq H(-v)^{2H-1} (t-s).
\end{eqnarray}
On the other hand, we have by (\ref{mal-sko})
\begin{equation}\label{ast}
D_{r}^{B_{i}} I^{B_{i}}(\phi) = \phi_{r} + I^{B_{i}}(D_{r}^{B_{i}}\phi)
=\big(\phi_r+ I^{B_{i}}({\bf 1}_{[r-v,+\infty)\cap [0,t-s]})\big)\1_{[0,t-s]}(r).
\end{equation}

When $-v\ge t-s$, then $[r-v,+\infty)\cap [0,t-s]=\emptyset$ for any $r\in[0,t-s]$.
By using (\ref{sk_dual}) we deduce
\begin{align}\label{eq1bis}
\E |I^{B^{i}}(\phi)|^{2} &= \E \| \phi \|_{\HH}^{2} =
\E \| B^i_{\cdot+v}-B^i_v \|_{\HH([0,t-s])}^{2} = \E \| B^i \|_{\HH([0,t-s])}^{2}  \nonumber \\ &=
(t-s)^{4H}\E\|B^i\|_{\HH([0,1])}^2,
\end{align} where
the two last equalities are due to  the stationarity (\ref{stationarity}) and scaling (\ref{scaling}) properties of fractional Brownian motion.

When $-v< t-s$, then
\begin{equation}\label{vla}
I^{B_{i}}({\bf 1}_{[r-v,+\infty)\cap [0,t-s]})=(B^i_{t-s}-B^i_{r-v}){\bf 1}_{[0,t-s+v]}(r).
\end{equation}
We deduce
\begin{eqnarray}
&&\E |I^{B^{i}}(\phi)|^{2} \nonumber\\
&=&\E \langle DI^{B^{i}}(\phi) ,\phi\rangle_\HH\quad\mbox{by (\ref{sk_dual})} \nonumber\\
&=&\E \| \phi \|_{\HH([0,t-s])}^{2}
+E \langle
I^{B^{i}}(
{\bf 1}_{[r-v,\infty)\cap[0,t-s]}),\phi\rangle_{\HH([0,t-s])}\quad\mbox{by (\ref{ast})}\nonumber\\
& =& \E \| \phi \|_{\HH([0,t-s])}^{2}
+E \langle (B^i_{t-s}-B^i_{\cdot-v}){\bf 1}_{[0,t-s+v]},\phi\rangle_{\HH([0,t-s])}\quad\mbox{by (\ref{vla})}\nonumber\\
& \le&
\E \| \phi \|_{\HH([0,t-s])}^{2}
+\E\left( \| (B^i_{t-s}-B^i_{\cdot-v}){\bf 1}_{[0,t-s+v]}\|_{\HH([0,t-s])} \,
 \| \phi \|_{\HH([0,t-s])}\right)\nonumber\\
&  \le&
\frac32 \E \| \phi \|_{\HH([0,t-s])}^{2}
+\frac12
\E \| (B^i_{t-s}-B^i_{\cdot-v}){\bf 1}_{[0,t-s+v]}\|_{\HH([0,t-s])}^2\quad \mbox{because
$ab\le\frac12(a^2+b^2)$}\nonumber\\
& =&
\frac32(t-s)^{4H}\E\|B^i\|_{\HH([0,1])}^2
+\frac12
\E \| (B^i_{t-s+v}-B^i){\bf 1}_{[0,t-s+v]}\|_{\HH([0,t-s])}^2\quad\mbox{by (\ref{scaling}) and (\ref{stationarity})}
\nonumber\\
& =&
\frac32(t-s)^{4H}\E\|B^i\|_{\HH([0,1])}^2
+\frac12(t-s+v)^{2H}
\E \| (B^i_{t-s+v}-B^i_{(t+s-v)\cdot})\|_{\HH([0,1])}^2
\nonumber\\
& =&
\frac32(t-s)^{4H}\E\|B^i\|_{\HH([0,1])}^2
+\frac12(t-s+v)^{4H}
\E \| (B^i_{1}-B^i)\|_{\HH([0,1])}^2\quad\mbox{by (\ref{scaling})}
\nonumber\\
& \leq &\frac12(t-s)^{4H} \left( 3 \E\|B^i\|_{\HH([0,1])}^2 + \E \| (B^i_{1}-B^i)\|_{\HH([0,1])}^2 \right) .\label{eq2bis}
\end{eqnarray}
Finally, we can summarize (\ref{eq1bis}) and (\ref{eq2bis}) in
$$
\E |I^{B^{i}}(\phi)|^{2}
\leq c_H |t-s|^{4H}, $$
with a constant $c_H>0$, in particular independent of $v$.
Putting together this last estimate with inequality (\ref{taylor_rem}),
we end up with:
$$
{\rm E} |\bdst(v)(i,i)|^2
\le c_H (1+ |v|^{2H-1}) |t-s|^{4H},
$$
for any $v\in[-r,0]$.

\vspace{0.3cm}

\noindent
{\it 2) Case where $i\neq j$}. By stationarity
  (\ref{stationarity}), we have for any $v\in[-r,0]$ that
$$
\big(B^j_{u+v}-B^j_{v},B^i_u\big)_{u\in[0,t-s]}
\,{\stackrel{{\cal L}}{=}}\,
 \big(B^j_{u},B^i_u\big)_{u\in[0,t-s]}.
$$
Thus, the delayed L\'evy area
$\bdot(v)(i,j)=\int_0^{t-s} (B^j_{u+v}-B^j_{v})d^\circ B_u^i$ for $v<0$ behaves
as in the case where $v=0$. But it is a classical result that $\bdot(0)$ is well-defined for
$H>1/3$ (see, e.g., \cite{PT2}). Moreover,
it follows again  by the
stationarity (\ref{stationarity}) and the scaling (\ref{scaling})
properties that
$$
{\rm E} |\bdot(v)(i,j)|^2
=
{\rm E} |\bdot(0)(i,j)|^2
= |t-s|^{4H} {\rm E} |{\bf B}^{\bf 2}_{01}(0)(i,j)|^2
\le c_H |t-s|^{4H}.
$$
Immediately, we deduce that
$$
{\rm E} |\bdst(v)(i,j)|^2
\le c_H |t-s|^{4H}
$$
for any $v\in[-r,0]$.
\vspace{0.3cm}

Both in the cases $i=j$ and $i\neq j$,
the substitution formula
for Russo-Vallois integrals easily
yields  that $\der\bd(v)=\der B^v\otimes\der B$. Furthermore,
since $\bd(v)$ is a process belonging to the
second chaos of the fractional Brownian motion $B$, on which all
$L^p$ norms are equivalent for $p>1$, we get that
\begin{equation}\label{ineq:norm-xd}
{\rm E} |\bdst(v)(i,j)|^p
\le c_p |t-s|^{2pH}
\end{equation}
for $i \neq j$ and
\begin{equation}\label{ineq:norm-xd2}
{\rm E}|I^{B_{i}}(\phi)|^p \le c_p |t-s|^{2pH}
\end{equation}
when $i=j$.
In order to conclude that $\bd(v)\in\cac_2^{2\ga}(\R^{d\times d})$ for any
$\frac13<\ga<H$ and $v\in[-r,0)$, let us recall the following inequality from \cite{Gu}:
let $g\in \cac_2(V)$ for a given Banach space $V$; then,
for any $\ka>0$ and $p\ge 1$ we have
\begin{equation}\label{lem:garsia}
\| g\|_{\ka}\le c \lp U_{\ka+2/p;p}(g) + \| \der g\|_{\ka}\rp
\quad\mbox{ with }\quad
U_{\ga;p}(g)=
\lp \int_0^T\int_0^T \frac{|g_{st}|^p}{|t-s|^{\ga p}}dsdt \rp^{1/p}.
\end{equation}
By plugging inequality (\ref{ineq:norm-xd})-(\ref{ineq:norm-xd2}) into (\ref{lem:garsia}),
by recalling that $\der\bd(v)=\der B^v\otimes\der B$ and
(\ref{taylor_rem}) hold, we obtain
 that $\bd(v)(i,j)\in\cac_2^{2\ga}(\R^{d\times d})$ for any $\frac13<\ga<H$ and
 $i,j=1, \ldots, d$.

\fin



\begin{thebibliography}{99}

\bibitem{ALN}
E. Al\`os, J. A. Le\'on and D. Nualart (2001):
\it Stratonovich stochastic calculus for fractional Brownian motion with Hurst parameter lesser that
$\frac12$.
\rm Taiwanese J. Math. {\bf 5}, 609-632.


\bibitem{al}
E. Al\'os and D. Nualart (2002):
\it Stochastic integration with respect to the fractional Brownian motion.
\rm Stochastics Stochastics Rep. {\bf 75}, 129-152.


\bibitem{cc}
P. Carmona, L. Coutin and G. Montseny (2003):
\it Stochastic integration with respect to fractional Brownian motion.
\rm Ann. Inst. H. Poincar\'e, Probab. Stat. {\bf 39},  27-68.


\bibitem{CN}
P. Cheridito and D. Nualart (2005):
\it Stochastic integral of divergence type with respect to fractional Brownian motion with
Hurst parameter $H$ in $(0,\frac12)$.
\rm Ann. Inst. H. Poincar\'e, Probab. Stat. {\bf 41}, 1049-1081.

\bibitem{CQ}
L. Coutin and Z. Qian (2002):
\it Stochastic rough path analysis and fractional Brownian motion.
\rm Probab. Theory Related Fields {\bf 122}, 108-140.


\bibitem{LD}
L. Decreusefond (2003):
\it Stochastic integration with respect to fractional Brownian motion.
\rm In: Doukhan, Paul (ed.) et al., Theory and applications of long-range dependence. Boston, MA: Birkh\"auser. 203-226.


\bibitem{FR}
M. Ferrante and C. Rovira (2006):
\it Stochastic delay differential equations driven by fractional Brownian motion with Hurst parameter
$H>\frac12$.
\rm Bernoulli {\bf 12} (1), 85–-100.

\bibitem{FP}
D. Feyel and A. de La Pradelle (2006):
\it Curvilinear integrals along enriched paths.
\rm Electron. J. Probab. {\bf 11}, 860-892.

\bibitem{Gu}
M. Gubinelli (2004):
\it Controlling rough paths.
\rm J. Funct. Anal. {\bf 216}, 86-140.



\bibitem{GT}
M. Gubinelli, S. Tindel (2007):
\it Rough evolution equation.
\rm In preparation.



\bibitem{lejay}
A. Lejay (2003):
\it An Introduction to Rough Paths.
\rm S\'eminaire de probabilit\'es 37, vol. {\bf 1832} of Lecture Notes in Mathematics, 1-59.


\bibitem{LyonsBook}
T. Lyons and Z. Qian (2002):
\it System control and rough paths.
\rm Oxford University Press.

\bibitem{Lyons}
T. Lyons (1998):
\it Differential equations driven by rough signals.
\rm Rev. Mat. Iberoamericana {\bf 14} (2), 215-310.





\bibitem{mohammed2} S.-E. A. Mohammed (1984):
\it Stochastic functional differential equations.
\rm Research Notes in Mathematics, 99. Boston-London-Melbourne: Pitman Advanced Publishing Program.


\bibitem{mohammed}
S.-E. A. Mohammed (1998):
\it Stochastic differential systems with memory: theory, examples and applications.
\rm In Stochastic Analysis and Related Topics VI (L. Decreusefond, J. Gjerde, B.
\O  ksendal
and A.S. \"Ust\"unel, eds), Birkh\"auser, Boston, 1-77.

\bibitem{NNRT}
A. Neuenkirch, I. Nourdin, A. R\"o\ss ler and S. Tindel (2006):
\it Trees and asymptotic developments for fractional diffusion processes.
\rm Preprint.

\bibitem{nourdin-simon}
I. Nourdin and T. Simon (2007):
\it Correcting Newton-Cotes integrals corrected by L\'evy areas.
\rm Bernoulli {\bf 13}, no. 3, 695-711.


\bibitem{nual-cours}
D. Nualart (2003):
\it Stochastic calculus with respect to the
fractional Brownian motion and applications.
\rm Contemp. Math. {\bf 336}, 3-39.

\bibitem{NR}
D. Nualart and  A. R\v{a}\c{s}canu (2002):
\it Differential equations driven by fractional Brownian motion.
\rm Collect. Math. {\bf 53} (1), 55-81.


\bibitem{PT2}
V. P\'erez-Abreu and C. Tudor (2002):
\it Transfer principle for stochastic fractional integral.
\rm Bol. Soc. Mat. Mexicana {\bf 8}, 55-71.


\bibitem{RV}
F. Russo and P. Vallois (1993):
\it Forward, backward and symmetric stochastic integration.
\rm Probab. Theory Related Fields  $\bf{97}$, 403-421.

\bibitem{ZA}
M. Z\"ahle (1998):
\it Integration with respect to fractal functions and stochastic  calculus I.
\rm Probab. Theory Related Fields {\bf 111}, 333-374.




\end{thebibliography}
\end{document}